\documentclass{article}
\usepackage{psfrag}
\usepackage[parfill]{parskip}
\usepackage{float, graphicx}
\usepackage[]{epsfig}
\usepackage{amsmath, amsthm, amssymb,}
\usepackage{epsfig}
\usepackage{verbatim}
\usepackage{multicol}
\usepackage{url}
\usepackage{latexsym}
\usepackage{mathrsfs}
\usepackage[colorlinks, bookmarks=true]{hyperref}
\usepackage{graphicx}
\usepackage{amsmath}
\usepackage{enumerate}
\usepackage[normalem]{ulem}
\usepackage{bm}
\usepackage{dsfont}
\usepackage{stmaryrd}
\usepackage{enumitem}
\usepackage{cite}
\usepackage{color}
\usepackage{xcolor}
\usepackage{hyperref}

\makeatletter

\def\@settitle{\begin{center}%
    \bfseries
 \normalfont\LARGE\@title
  \end{center}%
}
\def\@setauthors{\begin{center}%
 \normalsize\@author
  \end{center}%
}

\makeatother

\numberwithin{equation}{section}

\def\hatmfc{\hat{m}_{t, \mathrm{fc}}}
\def\tilmfc{\tilde{m}_{t, \mathrm{fc}}}
\def\d{\mathrm{d}}
\def\tilrhoinf{\tilde{\rho}_\infty}
\def\muHo{\mu_{H_0}}
\def\tilxi{\tilde{\xi}}
\def\hatxi{\hat{\xi}}
\def\e{\mathrm{e}}
\def\tilEp{\tilde{E}_+}
\def\tilxip{\tilde{\xi}_+}
\def\hatxip{\hat{\xi}_+}
\def\hatEp{\hat{E}_+}

\def\tilminf{\tilde{m}_\infty}
\def\O{\mathcal{O}}
\def\tilL{\tilde{L}}
\def\hatg{\hat{\gamma}_0}
\def\tilg{\tilde{\gamma}_0}
\def\rr{\mathbb{R}}
\def\ee{\mathbb{E}}
\def\hatgam{\hat{\gamma}_0}

\renewcommand{\cal}{\mathcal}
\newcommand\cA{{\mathcal A}}

\newcommand{\cC}{{\cal C}}
\newcommand{\cD}{{\cal D}}
\newcommand{\cE}{{\cal E}}

\newcommand{\cN}{{\cal N}}

\newcommand{\cR}{{\mathcal R}}

\newcommand{\cX}{{\mathcal X}}

\newcommand{\fa}{{\mathfrak a}}
\newcommand{\fb}{{\mathfrak b}}
\newcommand{\fc}{{\mathfrak c}}
\newcommand{\fd}{{\mathfrak d}}

\newcommand{\rd}{{\rm d}}

\newcommand{\ri}{\mathrm{i}}

\newcommand{\bC}{{\mathbb C}}
\newcommand{\bE}{\mathbb{E}}

\newcommand{\bP}{\mathbb{P}}

\newcommand{\bR}{{\mathbb R}}

\newcommand{\la}{\lambda}

\DeclareMathOperator{\Tr}{Tr}

\DeclareMathOperator{\supp}{supp}
\DeclareMathOperator{\dist}{dist}

\DeclareMathOperator{\OO}{O}
\DeclareMathOperator{\oo}{o}

\renewcommand{\Re}{\mathop{\mathrm{Re}}}
\renewcommand{\Im}{\mathop{\mathrm{Im}}}

\newcommand{\deq}{\mathrel{\mathop:}=} 
 
\renewcommand{\leq}{\leqslant}
\renewcommand{\geq}{\geqslant}

\newcommand{\nc}{\normalcolor}

\newcommand{\del}{\partial}

\newcommand{\beq}{\begin{equation}}
\newcommand{\eeq}{\end{equation}}

\theoremstyle{plain} 
\newtheorem{theorem}{Theorem}[section]
\newtheorem*{theorem*}{Theorem}
\newtheorem{lemma}[theorem]{Lemma}
\newtheorem*{lemma*}{Lemma}
\newtheorem{corollary}[theorem]{Corollary}
\newtheorem*{corollary*}{Corollary}
\newtheorem{proposition}[theorem]{Proposition}
\newtheorem*{proposition*}{Proposition}

\newtheorem*{assumption*}{Assumption}

\newtheorem{definition}[theorem]{Definition}
\newtheorem*{definition*}{Definition}

\newtheorem*{example*}{Example}
\newtheorem{remark}[theorem]{Remark}

\newtheorem*{remark*}{Remark}
\newtheorem*{remarks*}{Remarks}

\def\author#1{\par
    {\centering{\authorfont#1}\par\vspace*{0.05in}}
}

\def\titlefont{\fontsize{13}{15}\bfseries\boldmath\selectfont\centering{}}
\def\authorfont{\fontsize{13}{15}}

\let\affiliationfont\rhfont

\def\address#1{\par
    {\centering{\affiliationfont#1\par}}\par\vspace*{11pt}
}

\def\body{
\setcounter{footnote}{0}
\def\thefootnote{\alph{footnote}}
\def\@makefnmark{{$^{\rm \@thefnmark}$}}
}

\def\title#1{
    \thispagestyle{plain}
    \vspace*{-14pt}
    \vskip 79pt
    {\centering{\titlefont #1\par}}%
    \vskip 1em
}

\def\ev{\boldsymbol{e}}
\def\rhosc{\rho_{\text{sc}}}
\renewcommand{\i}{\mathrm{i}}

\begin{document}

\title{Transition from Tracy-Widom to Gaussian fluctuations of extremal eigenvalues of sparse Erd{\H o}s-R{\'e}nyi graphs}

\vspace{1.2cm}

\noindent \begin{minipage}[c]{0.33\textwidth}
 \author{Jiaoyang Huang}
\address{Harvard University\\
   E-mail: jiaoyang@math.harvard.edu}
 \end{minipage}
 \begin{minipage}[c]{0.33\textwidth}
 \author{Benjamin Landon}
\address{Harvard University\\
   E-mail: landon@math.harvard.edu}
 \end{minipage}
\begin{minipage}[c]{0.33\textwidth}
 \author{Horng-Tzer Yau}
\address{Harvard University \\
   E-mail: htyau@math.harvard.edu}

 \end{minipage}

~\vspace{0.3cm}

\begin{abstract}
We consider the statistics of the extreme eigenvalues of sparse random matrices, a class of random matrices that includes the normalized adjacency matrices of the Erd{\H o}s-R{\'e}nyi graph $G(N,p)$. Tracy-Widom fluctuations of the extreme eigenvalues for $p\gg N^{-2/3}$ was proved in \cite{Lee2016,MR2964770}. We prove that there is a crossover in the behavior
of the extreme eigenvalues at $p\sim N^{-2/3}$. In the case that $N^{-7/9}\ll p\ll N^{-2/3}$, we prove that the extreme eigenvalues have asymptotically Gaussian fluctuations. Under a mean zero condition and when $p=CN^{-2/3}$, we find that the fluctuations of the extreme eigenvalues are given by a combination of the Gaussian and the Tracy-Widom distribution. These results show that the eigenvalues at the edge of the spectrum of sparse Erd\H{o}s-R\'enyi graphs are less rigid than those of random $d$-regular graphs \cite{BHKY} of the same average degree.
\end{abstract}

{
\hypersetup{linkcolor=black}
\tableofcontents
}

\section{Introduction}\label{s:intro}

{\let\thefootnote\relax\footnote{The work of H.-T. Y. is partially supported by NSF Grant DMS-1307444,  DMS-1606305 and a Simons Investigator award.}}

In this work we study the statistics of eigenvalues at the edge of  the spectrum of sparse random matrices. A natural example is the  adjacency matrix of the  Erd\H{o}s-R\'enyi graph $G(N,p)$, which is the random undirected graph on $N$ vertices in which each edge appears independently with probability $p$. 

Introduced in \cite{MR0108839,MR0120167}, the Erd\H{o}s-R\'enyi graph $G(N,p)$ has numerous applications in  graph theory, network theory, mathematical physics and  combinatorics. For further information, we refer the reader to the monographs \cite{MR1782847,MR1864966}.
Many interesting properties of graphs are revealed by the eigenvalues and eigenvectors of their adjacency matrices. Such phenomena and the applications have been intensively investigated for over half a century. To mention some, we refer the readers to the books \cite{MR1421568, MR2882891} for a general discussion on spectral graph theory,  the survey article \cite{MR2247919} for the connection between eigenvalues and expansion properties of graphs, and the articles \cite{MR1468791, MR1240959, 609407,1Coifman24052005,2Coifman24052005,doi:10.1137/0611030,MR2952760, MR2294342,MR1450607} on the applications of eigenvalues and eigenvectors in various algorithms, i.e., combinatorial optimization, spectral partitioning and clustering.

The adjacency matrices of Erd\H{o}s-R\'enyi graphs have typically $pN$ nonzero entries in each column and are sparse if $p \ll 1$. 
When $p$ is of constant order, the Erd\H{o}s-R\'enyi matrix is essentially a Wigner matrix (up to a non-zero mean of the matrix entries).  When $p \to 0$ as $N \to \infty$, the law of the matrix entries is highly concentrated at $0$, and the Erd\H{o}s-R\'enyi matrix can be viewed as a singular Wigner matrix.
The singular nature of this ensemble can be expressed by the fact that the $k$-th moment of a matrix entry (in the scaling that the bulk of the eigenvalues lie in an interval of order $1$) decays like ($k \geq 2$)
\begin{align} \label{eqn:decaymom}
N^{-1} (p N)^{- (k-2)/2}.
\end{align}
When $p \ll 1$, this decay in $k$ is much slower than the $N^{-k/2}$ case of Wigner matrices, and is the main source of difficulties in studying sparse ensembles with random matrix methods.

The class of random matrices whose moments decay like \eqref{eqn:decaymom}  were introduced in the works \cite{MR3098073,MR2964770} as a natural generalization of the sparse Erd\H{o}s-R\'enyi graph and encompass many other sparse ensembles.  This is the class we study in this work.

The global statistics of the eigenvalues of the Erd\H{o}s-R\'enyi graph are well understood. The empirical eigenvalue distribution converges to the semi-circle distribution provided $p\gg 1/N$, which follows from Wigner's original proof. It was proven in \cite{MR1825319}, that the spectral norm converges to $2$ if and only if $p\gg \log N$. For the global eigenvalue fluctuations, it was proven in \cite{MR2953145} that the linear statistics, after normalizing by $p^{1/2}$, converge to a Gaussian random variable.


A three-step dynamical approach to the local statistics of random matrices has been developed in a series of papers
 \cite{MR2537522,MR2481753,MR3068390,MR2981427,MR2919197,MR2871147,MR2810797,kevin3,fix2,MR3687212,MR3541852,ajanki2}.  This strategy is as follows.
 \begin{enumerate}
\item Establish a local semicircle law controlling the number of eigenvalues in windows of size $\log(N)^C/N$.
\item Analyze the local ergodicity of Dyson Brownian motion to obtain universality after adding a small Gaussian noise to the ensemble.
\item A density argument comparing a general matrix to one with a small Gaussian component.
\end{enumerate}
 This approach has been used to prove the bulk universality of a wide variety of random matrix ensembles, e.g., the sparse random matrices considered here\cite{MR3098073,MR2964770,MR3429490} or random matrices with correlated entries \cite{MR3478311,MR3629874,correlated,LZ}. 
  A different  approach which proves universality for some class of random matrices  was developed in \cite{MR2784665}. 

Universality for the edge statistics of Wigner matrices (the statement that the distribution of the extremal eigenvalues converge to the Tracy-Widon law)   was first established by the moment method \cite{MR1727234} under certain symmetry assumptions on the distribution of the matrix elements. The moment method was further developed in \cite{MR2475670,MR2647136} and \cite{MR2726110}. A different approach to edge universality for Wigner matrices  based on the direct comparison with corresponding Gaussian ensembles was developed in \cite{MR2669449,MR2871147}.  Edge universality for sparse ensembles was proven first in the regime $p\gg N^{-1/3}$ in the works \cite{MR3098073,MR2964770} and then extended to the regime $p \gg N^{-2/3}$ in \cite{Lee2016}. 

One of the key obstacles in the proof of the edge universality for sparse ensembles is the lack of an optimal rigidity estimate at the edge of the spectrum, which is a overwhelming probability a-priori bound on the distance of the extremal eigenvalues from the spectral edge.  An important component of the work of \cite{Lee2016} is establishing such  optimal rigidity estimates at the edge; furthermore for sparse ensembles they calculate a deterministic correction to the usual semicircle spectral edge $\pm 2$.  This correction is larger than the $N^{-2/3}$ Tracy-Widom fluctuations and is therefore necessary for their proof of universality. 

In our first main result we show that a rigidity result can no longer hold on the scale $N^{-2/3}$ in the regime $N^{-7/9}\ll p\ll N^{-2/3}$. We find a contribution to the edge fluctuations of order $(p^{1/2} N)^{-1}$ which dominates the usual Tracy-Widom scale $N^{-2/3}$ for $p$ in this regime; moreover we find that these fluctuations are asymptotically Gaussian.  
Comparing this to the result of \cite{Lee2016} we see a transition in the behavior at $p=N^{-2/3}$ from the Tracy-Widom distribution to the Gaussian distribution. 

 Our proof is based on constructing a higher order self-consistent equation for the Stieltjes transform of the empirical eigenvalue distributions. 
Self-consistent equations  have been widely used in the random matrix theory \cite{MR2183281,MR0208649,HKR, AEK}. The leading order componennt of our self-consistent equation corresponds to the semicircle law.  
Higher order corrections to  the self-consistent equations 
or  the ``fluctuation averaging  lemma"  were also constructed in  \cite{BHKY, Lee2016,MR3119922,MR3085669}.  The works   \cite{BC, EvEc} calculate higher order corrections to the expected density of states for various random matrix ensembles. 
 In the work \cite{MR2871147} it was noted that naive power counting arguments can be improved using higher order cancellations in the trace of the Green's function.  The works \cite{Lee2016,schnelliedge1,schnelliedge2}  develop a systematic approach to exploiting this cancellation to arbitrarily high orders.

In the current setting, our derivation of  a  recursive moment estimate for the normalized trace  follows  the recent work of  \cite{Lee2016}. 
However, we construct an explicit {\it random} measure from carefully chosen observables of the random matrix which approximates the Stieltjes transform $m_N$ of the empirical measure of the random matrix.  We prove a rigidity result with respect to this random measure, and then show that the edge of this measure has asymptotic Gaussian fluctuations of a larger order than the rigidity estimate, implying the Gaussian fluctuations for the edge eigenvalues themselves. 

Our second main result concerns sparse random matrices $H$ with centered entries - e.g., $\bE[H_{ij}=\pm1/ \sqrt{2Np}]=p/2$ and $\bE[H_{ij}=0]=1-p$.  When $p=CN^{-2/3}$ we find that the distribution of the extreme eigenvalues converges to an independent sum of a Gaussian and Tracy-Widom random variable.  The Gaussian fluctuation we find in both the case $p \ll N^{-2/3}$ and $p = C N^{-2/3}$ is an  expansion in the support of the density of states - the correction to the smallest eigenvalue is the same as to the largest eigenvalue, but with the opposite sign.  We exhibit this correction as a specific extensive quantity involving the matrix elements of $H$.  If one subtracts this quantity one finds the usual Tracy-Widom fluctuations down to $p \gg N^{-7/9}$.  For example, we show that the gap $\lambda_2 - \lambda_1$ is still given asymptotically by the gap of the GOE, on the usual $N^{-2/3}$ scale.


In order to exhibit Tracy-Widom fluctuations, we compare sparse ensembles to Gaussian divisible ensembles - that is, a sparse ensemble with a GOE component.  In \cite{Lee2016} the local law allowed for the comparison of sparse ensembles directly to the GOE.  When $p \ll N^{-2/3}$ it does not appear possible to compare directly to the GOE and so instead we use a Gaussian divisible ensemble with a small $o(1)$ GOE component.

Edge universality for such ensembles is established in \cite{LY}.  This work proves a version of Dyson's conjecture on the local ergodicity of Dyson Brownian motion for edge statistics.  That is, for wide classes of initial data, the edge statistics of DBM coincides with the GOE, and moreover \cite{LY} finds the optimal time to equilibrium $t \sim N^{-1/3}$ for sufficiently regular initial data.

\emph{Organization.} We define the model and present the main results in the rest of Section \ref{s:intro}. 
In Section \ref{s:rigidity}, we obtain the optimal edge rigidity estimates in the regime $p\gg N^{-7/9}$. It implies that in the regime $N^{-7/9}\ll p \ll N^{-2/3}$, the extreme eigenvalues have Gaussian fluctuation. In Section \ref{sec:ht}, we recall the results from \cite{LY} for  edge universality of Gaussian divisible ensembles. 
In Section \ref{s:comparison} we analyze the Green function to compare a sparse ensemble to a Gaussian divisible ensemble and establish our results about Tracy-Widom fluctuations.

\emph{Conventions.} We use $C$ to represent large universal constants, and $c$ small universal constants, which may be different from line by line. Let $Y\geq 0$. 
We write $X\ll Y$ or $Y\gg X$, if there exists a small exponent $\fc>0$ such that $X\leq  N^{-\fc}Y$ for $N\geq N(\fc)$ large enough. We write $X\lesssim Y$ or $Y\gtrsim X$ if there exists a constant $C>0$, such that $X\leq CY$. 
We write $X\asymp Y$ if there exists a constant $C>0$ such that $Y/C\leq X\leq Y/C$.
We say an event $\Omega$ holds with overwhelming probability, if for any $D>0$, $\bP(\Omega)\geq 1-N^{-D}$ for $N\geq N(D)$ large enough.

\subsection{Sparse random matrices}

In this section we introduce the class of sparse random matrices that we consider.  This class was introduced in \cite{MR3098073,MR2964770} and we repeat the discussion appearing there.

The Erd\H{o}s-R\'enyi graph is the undirected random graph in which each edge appears with probability $p$. 
 It is notationally convenient to replace the parameter $p$ with $q$ defined through
\begin{align}
p = q^2/N.
\end{align}
We allow $q$ to depend on $N$.   We denote by $A$ the adjacency matrix of the Erd\H{o}s-R\'enyi graph.  $A$ is an $N\times N$ symmetric matrix whose entries $a_{ij}$ above the main diagonal are independent and distributed according to
\begin{align}
a_{ij} =  \begin{cases} 1 & \mbox{ with probability } q^2/N \\ 0 & \mbox{ with probability } 1 - q^2/N \end{cases}.
\end{align}

We extract the mean of each entry and rescale the matrix so that the limiting eigenvalue distribution is supported on $[-2,2]$.  We introduce the matrix $H$ by 
\begin{align}
H\deq \frac{A-q^2\vert \ev \rangle \langle \ev \vert
}{q\sqrt{1-q^2/N}}
\end{align}
where $\bm e$ is the unit vector
\begin{align}
\ev=  ( 1, ..., 1 )^T /\sqrt{N}.
\end{align}
%
%
It is easy to check that the matrix elements of $H$ have mean zero $\bE[h_{ij}]=0$, variance $\bE[h_{ij}^2]=1/N$, and satisfy the moment bounds 
\begin{align}
\bE [  h_{ij}^k ] = \frac{1}{ Nq^{k-2} }\left[\left(1-\frac{q^2}{N}\right)^{-k/2+1}\left(\left(1-\frac{q^2}{N}\right)^{k-1}+(-1)^k\left(\frac{q^2}{N}\right)^{k-1}\right)\right]=\frac{\Omega(1)}{Nq^{k-2}},
\end{align}
for $k\geq 2$.  This motivates the following defintion. 

\begin{definition}[Sparse random matrices] {\label{asup}}
We assume that   $H=(h_{ij})$ is an $N\times N$ random matrix whose entries are real and independent up to the symmetry constraint $h_{ij}=h_{ji}$. We further assume that $(h_{ij})$ satisfies $\bE[h_{ij}]=0$, $\bE[h_{ij}^2]=1/N$ and  that for any $k\geq 2$, the $k$-th cumulant of $h_{ij}$ is given by 
\begin{align}
\frac{\cal (k-1)!\cC_k}{Nq^{k-2}},
\end{align}
where $q=q(N)$ is the sparsity parameter, such that $0< q\lesssim \sqrt{N}$.  For $\cC_k$ we make the following assumptions.
\begin{enumerate}[label={\normalfont(\arabic*)}]
\item $| \cC_k | \leq C_k$ for some constant $C_k>0$. 
\item $ \cC_4 \geq c$ \label{item:4clower}
\end{enumerate}
\end{definition}
\begin{remark} The lower bound, $ \cC_4 \geq c$, ensures that the scaling by $q$ for the ensemble $H$ is ``correct.''
\end{remark}


We denote the eigenvalues of $H$ by
$
\lambda_1\geq \lambda_2\geq\cdots\geq \lambda_N,
$
and the Green's function of $H$ by
\begin{align*}
G(z):=(H-z)^{-1}.
\end{align*}
The Stieltjes transform of the empirical eigenvalue distribution is denoted by
\begin{align*}
m_N(z):=\frac{1}{N}\Tr G(z)=\frac{1}{N}\sum_{i=1}^N\frac{1}{\lambda_i-z}.
\end{align*}
%
%
For $N$-dependent random (or deterministic) variables $A$ and $B$, we say $B$ stochastically dominate $A$, if for any $\varepsilon>0$ and $D>0$, then
\begin{equation}
 \bP(A \geq N^\varepsilon B) \leq N^{-D},
\end{equation}
for $N\geq N(\varepsilon, D)$ large enough, and we write   $A \prec B$ or $A=\O_\prec(B)$.

We define the following quantity, which governs the fluctuation of the extreme eigenvalues
\begin{align}\label{e:defcX}
\cX\deq \frac{1}{N}\sum_{ij}\left(h_{ij}^2-\frac{1}{N}\right).
\end{align}
If $H$ is the normalized adjacency matrix of Er{\"o}ds-R{\'e}nyi graphs, $\cX$ characterizes the fluctuation of the total number of edges. $\cX$ is of size $\O_\prec(\sqrt{N}q)$. Moreover, by the central limit theorem, the fluctuations of $\sqrt{N}q\cX $ are asymptotically Gaussian with mean zero and variance $6\cC_4$. 

\subsection{Main Results}

We first recall the strong entrywise local semicircle law for sparse random matrices from \cite{MR3098073}.

\begin{theorem}{\normalfont \cite{MR3098073}}\label{locallaw}
Let $H$ be as in Definition \ref{asup}.  Let $\bf >0$ be a large constant. 
Then 
uniformly for any $z=E+\ri \eta$ such that $-\fb\leq E\leq \fb$ and $1/N\ll \eta\leq \fb$,
we have
\begin{align}\label{indLocLaw}
\max_{i,j}|G_{ij}(z)-\delta_{ij}m_{ sc}(z)|\prec\left(\frac{1}{q}+\sqrt{\frac{\Im m_{sc}(z)}{N\eta} }+\frac{1}{N\eta}\right),
\end{align}
where $G(z)$ is the Green's function of the matrix $H$, and $m_{sc}(z)$ is the Stieltjes transform of the semi-circle distribution.
\end{theorem}

Our first main results, Theorem \ref{t:gaussianfluc} and Corollary \ref{c:ERg}  concern the behavior of the extremal eigenvalues of $H$ and $A$ in the regime $N^{1/9} \ll q \ll N^{1/6}$.  Specifically, we prove that the extremal have Gaussian fluctuations governed by the 
 quantity $\cX$ as defined in  \eqref{e:defcX}. 

\begin{theorem}\label{t:gaussianfluc}
Let $H$ be as in Definition \ref{asup}, and $\cX$ be as defined in \eqref{e:defcX}. There exists a deterministic constant $L$ depending on $q$ and the cumulants $\cC_2, \cC_3, \cdots$ of $h_{ij}$ (as defined in Proposition \eqref{p:minfty}) so that the following holds. 
Let $\fc>0$ and an integer $k \geq 1$ be given.
 With overwhelming probability we have for $ 1 \leq i \leq k$, 
 \begin{align}\label{e:largeeig}
 |\lambda_i-L-\cX|\leq N^\fc \left(\frac{1}{q^6}+\frac{1}{N^{2/3}}\right),
 \end{align}
 and
 \begin{align}\label{e:smalleig}
 |\lambda_{N-i}+L+\cX|\leq N^\fc \left(\frac{1}{q^6}+\frac{1}{N^{2/3}}\right).
 \end{align}
\end{theorem}

As an easy corollary of Theorem \ref{t:gaussianfluc}, we derive asymptotic Gaussian fluctuations for the second largest eigenvalue of the Erd\H{o}s-R{\'e}nyi graph in the 
regime $N^{-7/9}\ll p\ll N^{-2/3}$.

\begin{corollary}\label{c:ERg}
Let $A$ be the adjacency matrix of the Erd{\H o}s-R{\'e}nyi graph $G(N,p)$. We denote its eigenvalues $\mu_1\geq \mu_2\geq \cdots\geq \mu_N$. Then for $N^{-7/9}\ll p\ll N^{-2/3}$, the second largest eigenvalue of $A$ converges weakly to a Gaussian random variable as $N \to \infty$, 
\begin{align}\label{e:ERg}
\sqrt{N}\left(\mu_2-\sqrt{Np}\left(2+(Np)^{-1}-\frac{5}{4}(Np)^{-2}\right)\right)\rightarrow \cN(0, 1).
\end{align}
Here, $\cN(0,1)$ is the standard Gaussian random variable.
\end{corollary}


The following theorem concerns Tracy-Widom fluctuations for sparse random matrices as defined in Definition \ref{asup}.  In the case $q = c N^{1/6}$ the limiting distribution of the extremal eigenvalues is given by an independent sum of a Tracy-Widom and Gaussian random variable.  In the case $N^{1/6} \ll q \ll N^{1/6}$ we recover the Tracy-Widom fluctuations after subtracting $\cX$.

\begin{theorem} \label{thm:Tracy-Widom}
Let $H$ be as in Definition \ref{asup} and  $N^{1/9}\ll q\leq C N^{1/6}$.  Let   $\cX$ be as defined in \eqref{e:defcX} and denote the eigenvalues of $H$ by $\la_1,\la_2,\cdots,\lambda_N$.  Let $F : \rr^k \to \rr$ be a bounded test function with bounded derivatives.  There is a universal constant $\fc>0$ so that,
\begin{align}\begin{split} \label{eqn:htedgeb1}
&\phantom{{}={}}\ee_{H}[ F (N^{2/3} ( \lambda_1 - L-\cX ), \cdots , N^{2/3} ( \lambda_k- L-\cX ) ]\\
& = \ee_{GOE}[ F (N^{2/3} ( \mu_1 - 2 ), \cdots, N^{2/3} ( \mu_k - 2) ) ]+\O\left(N^{-\fc}\right).
\end{split}\end{align}
The second expectation  is with respect to a GOE matrix with eigenvalues $\mu_1,\mu_2,\cdots, \mu_N$. In the case that $q=CN^{1/6}$ we have
\begin{align}\begin{split}\label{eqn:htedgeb2}
&\phantom{{}={}} \ee_{H}[ F (N^{2/3} ( \lambda_1 - L ), \cdots , N^{2/3} ( \lambda_k - L ) ] \\
&= \ee_{GOE,H}[ F (N^{2/3} ( \mu_1 - 2+\cX), \cdots, N^{2/3} ( \mu_k - 2 +\cX ) ) +\O\left(N^{-\fc}\right).
\end{split}\end{align}
In the second expectation, $H$ is independent of the GOE matrix and its eigenvalues $\mu_i$. 
Analogous results hold for the smallest eigenvalues.
\end{theorem}

\section{Edge Rigidity for Sparse Random Matrices}\label{s:rigidity}
In this section we prove the following rigidity estimates for sparse random matrices in the vicinity of the spectral edges. The following proposition states that the Stieltjes transform of the empirical eigenvalue density is close to the Stieltjes transform of a \emph{random} measure $\tilde \rho_\infty$, which is explicitly constructed in Proposition \ref{p:tminfty}.

Fix an arbitrarily small constant $\fa$ and denote the shifted spectral domain 
\begin{align}\label{e:defD}
\cD=\left\{z=\kappa+\ri\eta\in \bC_+:  | \kappa| \leq 1 ,0 \leq \eta\leq 1, |\kappa|+\eta\geq N^\fa\left(\frac{1}{q^3 N^{1/2}}+\frac{1}{q^3 N\eta}+\frac{1}{(N\eta)^2}\right)\right\}.
\end{align}

\begin{theorem}\label{thm:edgerigidity}
Let $H$ be as in Definition \ref{asup}. 
Let $m_N(z)$ be the Stieltjes transform of its eigenvalue density. There exists a deterministic constant $L$ (as defined in Proposition \eqref{p:minfty}), and an explicit random symmetric measure $\tilde \rho_\infty$ (as defined in Proposition \ref{p:tminfty}) with Stieltjes transform $\tilde m_\infty$. Fix $\tilde z = L+\cX+ z$, where $z=\kappa+\ri \eta$ and $\cX$ is defined in \eqref{e:defcX} so that the following holds. With overwhelming probability, uniformly for any $z\in \cD$, we have
\begin{itemize}
\item If $\kappa\geq 0$,
\begin{align}\label{e:stateoutS}
|m_N(\tilde z)-\tilde m_\infty(\tilde z)|\prec \frac{1}{\sqrt{|\kappa|+\eta}}\left(\frac{1}{N\eta^{1/2}}+\frac{1}{N^{1/2}q^{3/2}}+\frac{1}{(N\eta)^2}+\frac{1}{q^3 N\eta}\right),
\end{align}
\item If $\kappa\leq 0$,
\begin{align}\label{e:stateinS}
|m_N(\tilde z)-\tilde m_\infty(\tilde z)|\prec \frac{1}{N\eta}+\frac{1}{(N\eta)^{1/2}q^{3/2}},
\end{align}
\end{itemize}
The analogous statement holds for $\tilde z=-L-\cX+z$.
\end{theorem}

We postpone the proof of Theorem \ref{thm:edgerigidity} to Section \ref{s:construct} and \ref{s:highm} and \ref{s:proofrigidity}. 

\subsection{Proof of Theorem \ref{t:gaussianfluc} and Corollary \ref{c:ERg}}
In this section, we prove Theorem \ref{t:gaussianfluc} and Corollary \ref{c:ERg} which follow easily from Theorem \ref{thm:edgerigidity}. 
\begin{proof}[Proof of Theorem \ref{t:gaussianfluc}]
We only prove \eqref{e:largeeig}, as \eqref{e:smalleig} follows from the same argument.  Note that the spectral edge of the random measure $\tilde \rho_\infty$ satisfies \eqref{e:esttL}. 
By fixing a sufficiently small $\fc>0$ and taking 
\begin{align}
\eta=\frac{1}{N^{2/3}},\quad \kappa\geq N^{\fc}\left(\frac{1}{q^6}+\frac{1}{N^{2/3}}\right),
\end{align}
in  \eqref{e:stateoutS}, we get that
\begin{align}\label{e:mNbound}
|m_N(\tilde z)|=\O\left(\frac{\eta}{\sqrt{\kappa+\eta}}\right)+\O_{\prec}\left(\frac{1}{\sqrt{|\kappa|+\eta}}\left(\frac{1}{N\eta^{1/2}}+\frac{1}{N^{1/2}q^{3/2}}+\frac{1}{(N\eta)^2}+\frac{1}{q^3 N\eta}\right)\right)
\ll \frac{1}{N\eta},
\end{align}
with overwhelming probability, where we used \eqref{eqn:imtilminf}. The equation 
\eqref{e:mNbound} implies that there is no eigenvalue in the interval $[L+\cX+\kappa-\eta,L+\cX+\kappa+\eta]$. Otherwise,
if there was an eigenvalue, $\lambda_i$ in the spectral window $[L+\cX+\kappa-\eta,L+\cX+\kappa+\eta]$, we would have that
\begin{align}
\text{Im} [m_N(\tilde z)] \geq \frac{1}{N} \Im\left[\frac{1}{\lambda_i -(L+\cX+\kappa +\ri \eta)}\right] \geq \frac{1}{2N \eta},
\end{align}
contradicting \eqref{e:mNbound}. It follows that with overwhelming probability we have
\begin{align}\label{e:upperbound}
\lambda_1-L-\cX\leq N^{\fc}\left(\frac{1}{q^6}+\frac{1}{N^{2/3}}\right).
\end{align}

For the lower bound,  by the same argument as in \cite[Lemma B.1]{MR2639734} using the Helffer-Sj{\"o}strand formula, we have that with overwhelming probability
\begin{align}\label{e:lowerbound}
\#\{i:\lambda_i\in (L+\cX-\kappa, \infty))\}=\int_{L+\cX-\kappa}^{\infty}\rho_{\infty}(x)\rd x+ \oo\left(\kappa^{3/2}\right)
=\frac{2}{3\pi}\kappa^{3/2}+ \oo\left(\kappa^{3/2}\right),
\end{align}
where
\begin{align}
\kappa= N^{\fc}\left(\frac{1}{q^6}+\frac{1}{N^{2/3}}\right).
\end{align}
The claim \eqref{e:largeeig} follows from \eqref{e:upperbound} and \eqref{e:lowerbound} after taking $\fc >0$ sufficiently small.

\end{proof}

\begin{proof}[Proof of Corollary \ref{c:ERg}]
Consider the following normalized adjacency matrix 
\begin{align}
H\deq \frac{A-q^2\vert \ev \rangle \langle \ev \vert
}{q\sqrt{1-q^2/N}},
\end{align}
where $q=(Np)^{1/2}$. We denote the eigenvalues of $H$ by $\la_1\geq \la_2\geq\cdots\geq \la_N$.  It follows from Theorem \ref{t:gaussianfluc} that for any fixed index $i$,
\begin{align}\label{e:Gaussian}
\sqrt{N}q(\la_i-L)\rightarrow \cN(0,1), \quad L=2+\frac{1}{q^2}-\frac{5}{4q^4}.
\end{align}
By the Cauchy interlacing theorem,
\begin{align}\label{e:interlace}
\sqrt{1-q^2/N}\la_2\leq \mu_2\leq q\sqrt{1-q^2/N}\la_1,
\end{align}
and the claim \eqref{e:ERg} follows by combining \eqref{e:Gaussian}  and \eqref{e:interlace}.

\end{proof}


\subsection{Construction of the higher order self-consistent equation}
\label{s:construct}
In this section, we construct the higher order self-consistent equations for sparse random matrices. We also derive some useful estimates on the equilibrium measure and its Stieltjes transform. 

\begin{proposition}\label{p:DSE}
Let $H$ be as in Definition \ref{asup} and fix a large constant $\fb>0$. Let $m_N(z)$ be the Stieltjes transform of the empirical eigenvalue density of $H$. There exists a polynomial (the higher order self-consistent equation)
\begin{align}
P_0(z,m)=1+zm+Q(m),
\end{align}
where 
\begin{align}\label{e:defQ}
Q(m)=m^2 +  \frac{1}{q^2}\left(a_2 m^4 +a_3 m^6+\cdots\right),
\end{align}
is an even polynomial (of some high degree which depends on $q$) of $m$ with coefficients $a_2,a_3\cdots=\O(1)$ that depend on $q$ and the cumulants $\cC_2,\cC_3,\cdots$ of $h_{ij}$, such that  uniformly for any $z=E+\ri\eta$ with $-\fb\leq E\leq \fb$, $1/N\ll \eta\leq \fb$,
 we have
\begin{align}\label{e:defP0}
\bE[P_0(z, m_N(z))]\prec \frac{\bE[\Im[m_N(z)]]}{N\eta}.
\end{align}
\end{proposition}
The proof uses the cumulant expansion to compute the expectations, similar to the derivation of Stein's lemma. The cumulant expansion was used in \cite{MR3678478, Lee2016,BC} to study the spectrum of random matrices. We postpone the proof to the end of this section.

\begin{remark}
The same idea was first used in \cite{BHKY} to derive the higher order self-consistent equation for the adjacency matrices of $d$-regular graphs.
\end{remark}

\begin{remark}
The first few terms of $P_0$ are given by
\begin{align}\label{e:ms-ceqn}
P_0(z,m)=1+zm+m^2+\frac{6\cC_4}{q^2}m^4+\frac{120\cC_6}{q^4}m^6+\O\left(\frac{1}{q^6}\right).
\end{align}
\end{remark}

We recall $\cX$ as defined in \eqref{e:defcX}, which characterizes the fluctuation of the number of edges if $H$ is the normalized adjacency matrix of the Erd{\H o}s R{\'e}nyi graph. We define the random polynomial $P(z, m)$ as
\begin{align}
P(z, m)=P_0(z,m)+{\cX} m^2.
\end{align}

The solutions $m_{\infty}(z)$ of  $P(z, m_{\infty}(z))=0$ and  $\tilde m_{\infty}(z)$ of $P(z, \tilde m_{\infty}(z))=0$ define holomorphic functions from the upper half plane $\bC_+$ to itself. It turns out that they are Stieltjes transforms of probability measures $\rho_\infty$ and $\tilde \rho_\infty$ respectively. The following lemmas collect some properties of $m_{\infty}(z), \tilde m_{\infty}(z)$ and the measures $\rho_\infty$ and $\tilde \rho_\infty$.

\begin{proposition}\label{p:minfty}
There exists a deterministic algebraic function $m_{\infty}: \bC_+\rightarrow \bC_+$, which depends on the cumulants $\cC_2,\cC_3,\cdots$ of $h_{ij}$, such that the following holds:
\begin{enumerate}
\item $m_{\infty}$ is the solution of the polynomial equation, $P_0(z, m_{\infty}(z))=0$.
\item $m_{\infty}$ is the Stieltjes transform of a deterministic symmetric probability measure $\rho_{\infty}$. Moreover, $\supp \rho_{\infty}=[-L, L]$, where $L$ depends on the coefficients of $P_0$,  and $\rho_{\infty}$ is strictly positive on $(-L,L)$ and has square root behavior at the edge. 
\item We have the following estimate on the imaginary part of $m_{\infty}$, 
\begin{align} \label{eqn:imtilminf}
\Im[m_{\infty}(E+\ri \eta)]\asymp \left\{
\begin{array}{cc}
\sqrt{\kappa+\eta}, & \text{if $E\in [-L, L]$,}\\
\eta/\sqrt{\kappa+\eta}, & \text{if $E\not\in [- L, L]$,}
\end{array}
\right.
\end{align}
and 
\begin{align}
|\del_2 P_0(z,m_{\infty}(z))|\asymp \sqrt{\kappa+\eta},\quad 
\del_2^2 P_0(z,m_{\infty}(z))=1+\OO(1/q^2),
\end{align}
where $\kappa =\dist(\Re[z], \{- L,L\})$.
\end{enumerate}
\end{proposition}

\begin{proposition}\label{p:tminfty}
There exists an algebraic function $\tilde m_{\infty}: \bC_+\rightarrow \bC_+$, which depends on the random quantity $\cX$ as defined in \eqref{e:defcX}, such that the following holds:
\begin{enumerate}
\item $\tilde m_{\infty}$ is the solution of the polynomial equation, $P(z, \tilde m_{\infty}(z))=0$.
\item $\tilde m_{\infty}$ is the Stieltjes transform of a random symmetric probability measure $\tilde \rho_{\infty}$. Moreover, $\supp \tilde \rho_{\infty}=[-\tilde L, \tilde L]$ and $\rho_{\infty}$ is strictly positive on $(-\tilde L,\tilde L)$ and has square root behavior at the edge.
\item $\tilde L$ has Gaussian fluctuation, explicitly
\begin{align}\label{e:esttL}
\tilde L=L+\cX+\O_{\prec}\left(\frac{1}{\sqrt{N}q^{3}}\right).
\end{align}
We have the following estimate on the imaginary part of $\tilde m_{\infty}$,
\begin{align} \label{eqn:imtilminf}
\Im[\tilde m_{\infty}(E+\ri \eta)]\asymp \left\{
\begin{array}{cc}
\sqrt{\kappa+\eta}, & \text{if $E\in [-\tilde L, \tilde L]$,}\\
\eta/\sqrt{\kappa+\eta}, & \text{if $E\not\in [-\tilde L, \tilde L]$,}
\end{array}
\right.
\end{align}
and 
\begin{align}\label{e:derPest}
|\del_2 P(z,\tilde m_{\infty}(z))|\asymp \sqrt{\kappa+\eta},\quad
\del_2^2 P(z,\tilde m_{\infty}(z))=1+\OO(1/q^2),
\end{align}
where $\kappa =\dist(\Re[z], \{-\tilde L,\tilde L\})$.
\end{enumerate}
\end{proposition}

The proofs of Propositions \ref{p:minfty} and \ref{p:tminfty} follow the same proof as \cite[Lemma 4.1]{Lee2016}. For completeness, we sketch the proofs in  Appendix \ref{ap:pftminfty}. 

Before proving Proposition \ref{p:DSE}, we collect some estimates which will be used repeatedly in the rest of the paper.

\begin{proposition}
Let $H$ be as in Definition \ref{asup} and fix a large constant $\fb>0$. Let $G(z)$ be the Green's function of $H$ and $m_N(z)$ the Stieltjes transform of the eigenvalue density of $H$. For any indices $1\leq i,j\leq N$, we have
\begin{align}\label{e:hbound}
h_{ij}\prec \frac{1}{q}.
\end{align}
Uniformly for any $z=E+\ri\eta$ such that $-\fb\leq E\leq \fb$, $1/N\ll \eta\leq \fb$,
 we have
 \begin{align}\label{e:derivativeofG}
\sum_{k}\left|G_{ik}(z)G_{kj}(z)\right|\prec \frac{\Im[m_N(z)]}{\eta}
 \end{align}
\begin{align}\label{e:derivativeofm}
\left|\del_{ij}m_N(z)\right|\prec \frac{\Im[m_N(z)]}{N\eta},\quad 
\frac{1}{N}\left|\del_z m_N(z)\right|\prec \frac{\Im[m_N(z)]}{N\eta},
\end{align}
where $\del_{ij}$ is the derivative with respect to $h_{ij}$.
\end{proposition}
\begin{proof}
The bound on $h_{ij}$ follows from Markov's equality. For \eqref{e:derivativeofG}, by the Cauchy-Schwartz inequality and the Ward identity \eqref{e:Ward}, we have
\begin{align}
\left| \sum_{k}G_{ik}G_{kj}\right|\leq \frac{1}{2}\sum_{k}\left[|G_{kj}|^2+|G_{ik}|^2\right]=\frac{\Im[G_{jj}]+\Im[G_{ii}]}{2\eta}\prec \frac{\Im[m_N]}{\eta},
\end{align}
where in the last inequality we used \eqref{e:offbound2}. The estimates in \eqref{e:derivativeofm} follow from \eqref{e:derivativeofG} and the calculations,
\begin{align}
\left|\del_{ij}m_N\right|=\frac{2-\delta_{ij}}{N}\left|\sum_{k}G_{ik}G_{kj}\right|\prec \frac{\Im[m_N]}{N\eta},
\end{align}
and
\begin{align}
\frac{1}{N}\left|\del_{z}m_N\right|=\frac{1}{N^2}\left|\sum_{ij}G_{ij}G_{ji}\right|\prec \frac{\Im[m_N]}{N\eta}.
\end{align}
\end{proof}

The following proposition allows us to estimate the expectation of certain polynomials of entries of the Green's function.
\begin{proposition}\label{p:one-off}
Let $H$ be as in Definition \ref{asup} and fix a large constant $\fb>0$. Let $G(z)$ be the Green's function of the matrix $H$. Let $F(G(z))$ be a function of the following form,
\begin{align}
F(G(z))=\prod_{k=1}^d \left(\frac{1}{N}\sum_{i=1}^NG_{ii}^{s_k}(z)\right).
\end{align}
Uniformly for any $z=E+\ri\eta$ such that $-\fb\leq E\leq \fb$, $1/N\ll \eta\leq \fb$,
we have 
\begin{align}
\left|\frac{1}{N^2}\bE\left[\sum_{ij}G_{ij}G_{ii}^{r_1}G_{jj}^{r_2}F(G)\right]\right|\prec \frac{\bE[\Im[m_N]]}{N\eta}.
\end{align}
\end{proposition}

\begin{proof}
For the sum of diagonal entries, we have the following trivial bound
\begin{align}
\left|\frac{1}{N^2}\sum_{i=j}G_{ij}G_{ii}^{r_1}G_{jj}^{r_2}F(G)\right|\prec \frac{1}{N}.\end{align}
For the off-diagonal terms, we use the following identities
\begin{align}
G_{ij}=G_{ii}\sum_{k\neq i}h_{ik}G_{kj}^{(i)},\quad G_{kj}^{(i)}=G_{kj}-\frac{G_{ki}G_{ij}}{G_{ii}},
\end{align}
and the cumulant expansion. We take a large constant $l$ such that $q^\ell \geq N^{10}$, 
\begin{align}\begin{split}
\frac{1}{N^2}\bE\left[\sum_{ij}G_{ij}G_{ii}^{r_1}G_{jj}^{r_2}F(G)\right]
&=\frac{1}{N^2}\sum_{ijk}\bE\left[h_{ik}G_{kj}^{(i)}G_{ii}^{r_1+1}G_{jj}^{r_2}F(G)\right]\\
&=\sum_{p=1}^{\ell}\frac{\cC_{p+1}}{N^3q^{p-1}}\sum_{ijk}\bE\left[G_{kj}^{(i)}\del_{ik}^p\left(G_{ii}^{r_1+1}G_{jj}^{r_2}F(G)\right)\right]+\O\left(\frac{1}{N^3}\right)\\
&=\sum_{p=1}^{\ell}\frac{\cC_{p+1}}{N^3q^{p-1}}\sum_{ijk}\bE\left[G_{kj}\del_{ik}^p\left(G_{ii}^{r_1+1}G_{jj}^{r_2}F(G)\right)\right]
+\O_\prec\left(\frac{\bE[\Im[m_N]]}{N\eta}\right).
\end{split}\end{align}
We consider terms in the Leibniz expansion of $\del_{ik}^p\left(G_{ii}^{r_1+1}G_{jj}^{r_2}F(G)\right)$.  If such a term contains at  least one off-diagonal term, then the resulting expression $G_{kj}\del_{ik}^p(G_{ii}^{r_1+1}G_{jj}^{r_2}F(G))$ contains at least two off-diagonal terms. We can estimate it using \eqref{e:derivativeofG}, which gives an error $\O_{\prec}(\Im[m_N]/N\eta)$. For example if $A_{ijk}$ is a term in the expansion of  $\del_{ik}^p\left(G_{ii}^{r_1+1}G_{jj}^{r_2}F(G)\right)$ that contains the  off-diagonal term $G_{ik}$, then
\begin{align}\begin{split}\label{e:cauchybound}
\left|\frac{\cC_{p+1}}{N^3q^{p-1}}\sum_{ijk}\bE\left[G_{kj} A_{ijk} 
\right]\right|&\prec \frac{1}{N^3}\sum_{ijk}\bE\left[|G_{kj}G_{ik}|\right]\\
&\leq  \frac{1}{N^3}\sum_{ijk}\bE\left[|G_{kj}|^2+|G_{ik}|^2\right]\prec \frac{\bE[\Im[m_N]]}{N\eta}.
\end{split}\end{align}
Thus,
\begin{align}\begin{split}\label{e:oneoffmain}
&\sum_{p\geq 1}\frac{\cC_{p+1}}{N^3q^{p-1}}\sum_{ijk}\bE\left[G_{kj}\del_{ik}^p\left(G_{ii}^{r_1+1}G_{jj}^{r_2}F(G)\right)\right]\\
=&\sum_{p\geq 1}\frac{(r_1+1)\cC_{p+1}}{N^3q^{p-1}}\sum_{ijk}\bE\left[G_{kj}\left(\del_{ik}^pG_{ii}^{r_1+1}\right)G_{jj}^{r_2}F(G)\right]+\O_{\prec}\left(\frac{\bE[\Im[m_N]]}{N\eta}\right).
\end{split}\end{align}
Note that when derivatives hit the $F(G)$ term they generically contain off-diagonal terms or are of lower cardinality. 
If $p$ is an odd number then $(\del_{ik}^{p}G_{ii}^{r_1+1})$ must contain off-diagonal terms and so we dispense with this case. If $p$ is an even number, then 
\begin{align}\label{e:dikG}
\del_{ik}^{p}G_{ii}^{r_1+1}=p! {p/2+r_1\choose p/2}G_{ii}^{p/2+r_1+1}G_{kk}^{p/2}+\text{terms with off-diagonal terms}.
\end{align}
For terms in $(\del_{ik}^{p}G_{ii}^{r_1+1})$ which contains off-diagonal terms, we can estimate them in the same way as \eqref{e:cauchybound}, which gives an error $\O_{\prec}(\Im[m_N]/N\eta)$. Therefore, by plugging \eqref{e:dikG} into \eqref{e:oneoffmain}, we get 
\begin{align}
\frac{1}{N^2}\bE\left[\sum_{ij}G_{ij}G_{ii}^{r_1}G_{jj}^{r_2}F(G)\right]
&= \sum_{p\geq 1,2|p}\frac{p!\cC_{p+1}}{N^2q^{p-1}}{p/2+r_1\choose p/2}\sum_{jk}\bE\left[G_{kj}G_{kk}^{p/2}G_{jj}^{r_2}\left(\frac{1}{N}\sum_{i}G_{ii}^{r_1+p/2+1}\right)F(G)\right]\nonumber\\
\label{e:expand1}&+\O_{\prec}\left(\frac{\bE[\Im[m_N]]}{N\eta}\right).
\end{align}
The terms on the righthand side of \eqref{e:expand1} are in the same form as the one on the lefthand side of \eqref{e:expand1}, except that there are some factors of $1/q$ in front of them. For each term on the righthand side of \eqref{e:expand1}, we can repeat the procedure above, each time gaining at least a factor of $q^{-1}$ for any terms that contain only diagonal Green's function elements.   After finitely many steps all the errors are bounded by $\O_\prec(N^{-1})$, and the claim follows.
\end{proof}

\begin{proof}[Proof of Proposition \ref{p:DSE}]
The polynomial $P_0(z,m)$ is constructed in a way such that \eqref{e:defP0} is satisfied. More precisely,
we compute the following expectation using the cumulant expansion
\begin{align}
\bE[1+zm_N]
=\frac{1}{N}\sum_{ij}\bE[h_{ij}G_{ij}]
=\frac{1}{N}\sum_{ij}\sum_{p=1}^{\ell}\frac{\cC_{p+1}}{Nq^{p-1}}\bE[\del_{ij}^{p}G_{ij}]+\O\left(\frac{1}{q^{\ell}}\right),
\end{align}
where $\ell$ is large enough such that $q^\ell\geq N^{10}$.
For the derivatives of the resolvent entries $G_{ij}$, three kinds of terms might occur: the terms containing more than one (counting multiplicity) off-diagonal entries, the terms containing one (counting multiplicity) off-diagonal entry, or the terms containing no off-diagonal entry.

For those terms containing more than one off-diagonal entry, using \eqref{e:derivativeofG}, we have the following trivial bound,
\begin{align}\label{e:boundtwooff}
\left|\frac{1}{N^2}\sum_{ij}\bE\left[G_{ij}^{r_1}G_{ii}^{r_2}G_{jj}^{r_3}\right]\right|\prec\frac{1}{N^2}\sum_{ij}\bE\left[|G_{ij}|^{2}\right]\prec \frac{\bE[\Im[m_N]]}{N\eta}.
\end{align}
For those terms containing one off-diagonal entry, thanks to Proposition \ref{p:one-off}, we have the same estimate
\begin{align}
\left|\frac{1}{N^2}\sum_{ij} \bE\left[G_{ij}G_{ii}^{r_1}G_{jj}^{r_2}\right]\right|\prec \frac{\bE[\Im[m_N]]}{N\eta}.
\end{align}

Finally we analyze those terms containing no off-diagonal entry, say $G_{ii}^{r_1}G_{jj}^{r_2}$. If $r_1,r_2\leq 1$, then its average can be written in terms of $m_N$, i.e.
\begin{align}
\frac{1}{N^2}\sum_{ij}G_{ii}^{r_1}G_{jj}^{r_2}=m_N^{r_1+r_2}.
\end{align}
Otherwise, we assume $r_1\geq 2$. Thanks to the following identities,
\begin{align}\begin{split}
G_{ii}^{r_1}G_{jj}^{r_2}+zG_{kk}G_{ii}^{r_1}G_{jj}^{r_2}= &\left(\sum_{m}h_{km}G_{km}\right)G_{ii}^{r_1}G_{jj}^{r_2},\\
G_{kk}G_{ii}^{r_1-1}G_{jj}^{r_2}+zG_{kk}G_{ii}^{r_1}G_{jj}^{r_2}=& \left(\sum_{m}h_{im}G_{im}\right)G_{kk}G_{ii}^{r_1-1}G_{jj}^{r_2},
\end{split}\end{align}
by taking difference, we get
\begin{align*}
&\bE[G_{ii}^{r_1}G_{jj}^{r_2}]
=\bE[G_{kk}G_{ii}^{r_1-1}G_{jj}^{r_2}]
+\sum_{m}\bE\left[h_{km}G_{km}G_{ii}^{r_1}G_{jj}^{r_2}\right]
-\sum_{m}\bE\left[h_{im}G_{im}G_{kk}G_{ii}^{r_1-1}G_{jj}^{r_2}\right]\\
=&\bE[G_{kk}G_{ii}^{r_1-1}G_{jj}^{r_2}]
+\sum_{m}\sum_{p=1}^{\ell}\frac{\cC_{p+1}}{Nq^{p-1}}\left(\bE\left[\del^p_{km}\left(G_{km}G_{ii}^{r_1}G_{jj}^{r_2}\right)\right]
-\bE\left[\del^p_{im}\left(G_{im}G_{kk}G_{ii}^{r_1-1}G_{jj}^{r_2}\right)\right]\right)+\O\left(\frac{1}{q^{\ell}}\right).
\end{align*}
We consider first the leading term in the cumulant expansion $p=2$.  The crucial point is that  the only terms resulting from the Leibniz expansion of the derivatives which do not contain at least two off-diagonal entries directly cancel:
\begin{align*}
\begin{split}
&=\frac{\cC_{2}}{N}\left(\bE\left[\del_{km}\left(G_{km}G_{ii}^{r_1}G_{jj}^{r_2}\right)\right]
-\bE\left[\del_{im}\left(G_{im}G_{kk}G_{ii}^{r_1-1}G_{jj}^{r_2}\right)\right]\right)\\
&=\frac{\cC_{2}}{N}\left(\bE\left[\del_{km}\left(G_{km}\right)G_{ii}^{r_1}G_{jj}^{r_2}\right]
-\bE\left[\del_{im}\left(G_{im}\right)G_{kk}G_{ii}^{r_1-1}G_{jj}^{r_2}\right]\right)+\text{terms with at least two off-diagonal terms}\\
&=\frac{\cC_{2}}{N}\left(
\bE\left[\left(G_{ii}G_{mm}\right)G_{kk}G_{ii}^{r_1-1}G_{jj}^{r_2}\right]-\bE\left[\left(G_{kk}G_{mm}\right)G_{ii}^{r_1}G_{jj}^{r_2}\right]\right)+\text{terms with at least two off-diagonal terms}\\
&=\text{terms with at least two off-diagonal terms}.
\end{split}\end{align*}
The terms with at least two off-diagonal terms can be bounded as in \eqref{e:boundtwooff}. Therefore, averaging over $i,j,k,m$, we obtained that 
\begin{align}
\frac{1}{N^2}\sum_{ij}\bE[G_{ii}^{r_1}G_{jj}^{r_2}]=\frac{1}{N^3}\sum_{ijk}\bE[G_{kk}G_{ii}^{r_1-1}G_{jj}^{r_2}]+\text{higher order terms}+\O_\prec \left(\frac{\bE[\Im[m_N(z)]]}{N\eta}\right),
\end{align}
where those higher order terms refer to terms of size $\O(q^{-1})$. In this way we introduced a new resolvent entry $G_{kk}$, and the exponent of $G_{ii}$ was reduced by one. By repeating this process, we can reduce all the exponents to one. The average of the resulting expression can be written in terms of $m_N$, i.e. $m_N^{r_1+r_2}$. We can repeat the above procedure to the higher order terms, until the error is of order $\O_\prec(N^{-1})$. In this way, we obtain a polynomial $Q(m_N)$ such that
\begin{align}
\bE[1+zm_N]=-\bE[Q(m_N)]+\O_{\prec}\left(\frac{\bE[\Im[m_N]]}{N\eta}\right).
\end{align} 
\end{proof}

\subsection{Moments of the higher order self-consistent equation}
\label{s:highm}

%

In this section we compute higher moments of the self-consistent equation near the spectral edges $\pm \tilde L$.  This gives us a recursive moment estimate for the Stieltjes transform $m_N(z)$.  The rigidity estimates follow from a careful analysis of the recursive moment estimate and an iteration argument.

Note that the polynomial we consider is random and depends on $\cX$. We will need to use the deterministic bounds of the random quantity in \eqref{eqn:imtilminf}, and so we have to consider the Stieltes transform $m_N$ at the random spectral argument $\tilde z=\pm(L+\cX)+z$,
where $z\in \cD$. We recall the shifted spectral domain from \eqref{e:defD}
\begin{align}\label{e:recalldefD}
\cD=\left\{z=\kappa+\ri\eta\in \bC_+: |\kappa| \leq 1,  0\leq \eta\leq 1, |\kappa|+\eta\geq N^\fa\left( \frac{1}{q^3 N^{1/2}}+\frac{1}{q^3 N\eta}+\frac{1}{(N\eta)^2}\right)\right\}.
\end{align}

\begin{proposition}\label{t:Pmoment}
Let $H$ be as in Definition \ref{asup}. Let $m_N(z)$ be the Stieltjes transform of its eigenvalue density. 
We recall $L$ as defined in Proposition \ref{p:minfty}, and $\tilde \rho_\infty$ with its Stieltjes transform $\tilde m_\infty$ as defined in Proposition \ref{p:tminfty}.
Fix $\tilde z = L+\cX+ z$, where $z=\kappa+\ri \eta\in \cD$ and $\cX$ is defined in \eqref{e:defcX}.
We  have the estimate 
\begin{align}\begin{split}\label{e:Pmoment}
\phantom{{}+{}}\bE[|P(\tilde z,\tilde m_{N}(\tilde z)|^{2r}]
\prec \max_{s_1+s_2\geq1}\bE\left[\left\{ \left(\frac{1}{q^3}+\frac{1}{N\eta}\right) \frac{\Im[m_N(\tilde z)]|\del_2 P(\tilde z, m_N(\tilde z))|}{N\eta} \right \} ^{s_1/2}\right.&\\
\left.\left(\frac{\Im[m_N(\tilde z)]}{N\eta}\right)^{s_2}|P(\tilde z, m_N(\tilde z))|^{2r-s_1-s_2}\right]&.
\end{split}\end{align}
The analogous statement holds for $\tilde z=-L-\cX+z$.
\end{proposition}

Before proving Proposition \ref{t:Pmoment}, we first  derive some estimates of the derivatives of the polynomial $P$ and the Green's function $G$ with respect to the matrix entries of $H$.

\begin{proposition}\label{p:derPandG}
Using the same notations as in Proposition \ref{t:Pmoment}, for the derivatives of $P(\tilde z, m_N(\tilde z))$, we have
\begin{align}\label{e:hderP}
|\del_{ij}^pP(\tilde z, m_N(\tilde z))|\prec |\del_{2}P(\tilde z, m_N(\tilde z))|\frac{\Im[m_N(\tilde z)]}{N\eta}.
\end{align}
For the derivatives of $G_{ij}(\tilde z)$, we have
\begin{align}\label{e:hderG}
\del_{ij}^pG_{ij}(\tilde z)=(\del_{ij}^pG_{ij})(\tilde z)+\O_{\prec}\left(\frac{\Im[m_N(\tilde z)]}{N\eta}\right).
\end{align}

\end{proposition}

\begin{proof}
In this proof, we write $P=P(\tilde z, m_N(\tilde z))$ and $\del_2 P= (\del_{2}P)(\tilde z,  m_N(\tilde z))$. We recall $Q$ as defined in \eqref{e:defQ}. For the derivative of $P$,
\begin{align}\begin{split}\label{e:derP}
\del_{ij}P
&=\del_{ij}(1+\tilde z m_N(\tilde z)+Q(m_N(\tilde z))+\cX m_N^2(\tilde z))
=
\del_2 P\del_{ij}(m_N(\tilde z))+(\del_{ij}\tilde z)m_N(\tilde z)+(\del_{ij}\cX)m_N^2(\tilde z)\\
&=\del_2 P(\del_{ij}m_N)(\tilde z)+\frac{2-\delta_{ij}}{N}h_{ij}(m_N(\tilde z)+m_N^2(\tilde z)+\del_2 P\del_zm_N(\tilde z))
\prec |\del_2 P|\frac{\Im[m_N(\tilde z)]}{N\eta},
\end{split}\end{align}
where we used \eqref{e:derivativeofm}.
Similarly, for the higher order derivatives of $P$, we have
\begin{align}\label{e:hderP}
|\del_{ij}^pP(\tilde z, m_N(\tilde z))|\prec |\del_2 P|\frac{\Im[m_N(\tilde z)]}{N\eta}.
\end{align}
For the derivative of $G_{ij}(\tilde z)$, we have
\begin{align}\begin{split}
\del_{ij}G_{ij}(\tilde z)
&=\del_{ij}(G_{ij})(\tilde z)+(\del_z G_{ij} )(\tilde z) \del_{ij}\tilde z\\
&=\del_{ij}(G_{ij})(\tilde z)+\frac{2-\delta_{ij}}{N}h_{ij}\sum_{k}G_{ik}(\tilde z)G_{kj}(\tilde z)
=\del_{ij}(G_{ij})(\tilde z)+\O_{\prec}\left(\frac{\Im[m_N(\tilde z)]}{N\eta}\right),
\end{split}\end{align}
where in the last equality we used \eqref{e:derivativeofG}. The higher order derivatives of $G_{ij}(\tilde z)$ can be estimated in the same way.
\end{proof}

\begin{proof}[Proof of Proposition \ref{t:Pmoment}]
%
%
In this proof we write, for simplicity of notation, $P=P(\tilde z, m_N(\tilde z))$, $G=G(\tilde z)$, $m_N=m_N(\tilde z)$, $P'=P'(\tilde z, m_N(\tilde z))=\del_{2}P(\tilde z, m_N(\tilde z))$, $\del_{ij}^p P=\del_{ij}^p(P(\tilde z, m_N(\tilde z)))$, $\del_{ij}^pG=\del_{ij}^p(G(\tilde z))$, $\del_{ij}^pm_N=\del_{ij}^p(m_N(\tilde z))$, $D_{ij}^pG=(\del_{ij}^pG)(\tilde z)$ and $D_{ij}^pm_N=(\del_{ij} ^p m_N)(\tilde z)$.

The starting point is the following identity,
\begin{align}
1+\tilde z m_N(\tilde z)=\sum_{ij}h_{ij}G_{ij}(\tilde z).
\end{align}
Using the cumulant expansion, we can write the moment of $P(\tilde z, m_N(\tilde z))$ as,
\begin{align}
\begin{split}\label{e:Pmoment}
&\phantom{{}={}}\bE[|P(\tilde z,m_N(\tilde z))|^{2r}]
=\bE[Q P^{r-1}\bar P^{r}]+\bE[\cX m_N^2P^{r-1}\bar P^{r}]+\frac{1}{N}\bE\left[\sum_{ij}h_{ij}G_{ij}(\tilde z)P^{r-1}\bar P^{r}\right]\\
&=\bE[Q P^{r-1}\bar P^{r}]+\bE[\cX m_N^2P^{r-1}\bar P^{r}]+\frac{1}{N}\sum_{ij}\sum_{p=1}^{\ell}\frac{\cC_{p+1}}{Nq^{p-1}}\bE[\del_{ij}^{p}(G_{ij}P^{r-1}\bar P^r)]+\O\left(\frac{1}{q^{\ell}}\right),
\end{split}
\end{align}
where $\ell$ is large enough such that $q^\ell\geq N^{10}$ (say).
In the following we estimate the sum on the righthand side of \eqref{e:Pmoment}. For $p\neq 2,3$, we write that
\begin{align}
\bE[\del_{ij}^p(G_{ij}P^{r-1}\bar P^r)]
=\bE[D_{ij}^p(G_{ij})P^{r-1}\bar P^r]+\text{other terms},
\end{align}
and show the leading order term cancels with $\bE[Q P^{r-1}\bar P^{r}]$ from our construction of $Q$, and the ``other terms" are negligible. For $p=2,3$, we write 
\begin{align}
\bE[\del_{ij}^p(G_{ij}P^{r-1}\bar P^r)]
=\bE[D_{ij}^p(G_{ij})P^{r-1}\bar P^r]-p\bE[G_{ii}G_{jj}\del_{ij}^{p-1}(P^{r-1}\bar P^r)]+\text{other terms},
\end{align}
 and the ``other terms" are negligible. We will show that 
 the leading order contribution  of the first term on the right hand side   cancels with $\bE[Q P^{r-1}\bar P^{r}]$, and  
that of  the second term cancels  with the  fluctuation term $\bE[\cX m_N^2P^{r-1}\bar P^{r}]$ in \eqref{e:Pmoment}. 

For the first order term in \eqref{e:Pmoment}, i.e. $p=1$,
\begin{align}
\begin{split}\label{e:p=1}
\frac{1}{N^2}\sum_{ij}\bE[\del_{ij}(G_{ij}P^{r-1}\bar P^r)]
=\frac{1}{N^2}\sum_{ij}\bE[\del_{ij}(G_{ij})P^{r-1}\bar P^r]
+\frac{1}{N^2}\sum_{ij}\bE[G_{ij}\del_{ij}(P^{r-1}\bar P^r)]
\end{split}
\end{align}
For the first term in \eqref{e:p=1}, thanks to \eqref{e:hderG}
\begin{align}
\begin{split}
&\phantom{{}+{}}\frac{1}{N^2}\sum_{ij}\bE[\del_{ij}(G_{ij})P^{r-1}\bar P^r]
=-\frac{1}{N^2}\sum_{ij}\bE[G_{ii}G_{jj}P^{r-1}\bar P^r]-\frac{1}{N^2}\sum_{ij}\bE[G_{ij}G_{ij}P^{r-1}\bar P^r]\\
&+\O_{\prec}\left(\bE\left[\frac{\Im[m_N]}{N\eta}|P|^{2r-1}\right]\right)=-\bE[m_N^2P^{r-1}\bar P^r]+\O_{\prec}\left(\bE\left[\frac{\Im[m_N]}{N\eta}|P|^{2r-1}\right]\right),
\end{split}
\end{align}
where we used \eqref{e:derivativeofG}.
%
For the second term in \eqref{e:p=1}, we use Proposition \ref{p:derPandG}. We write the second term in \eqref{e:p=1} as follows
\begin{align}
\begin{split}\label{e:p=1_2}
\frac{1}{N^2}\sum_{ij}\bE[G_{ij}\del_{ij}(P^{r-1}\bar P^r)]
=\frac{\O(1)}{N^2}\sum_{ij}\bE[G_{ij}(\del_{ij}P) P^{r-2}\bar P^r]+\frac{\O(1)}{N^2}\sum_{ij}\bE[G_{ij}\del_{ij}(\bar P) |P|^{2r-2}]
\end{split}
\end{align}
Using \eqref{e:derP}, the first term in \eqref{e:p=1_2} is
\begin{align}
\begin{split}
&\phantom{{}={}}\frac{\O(1)}{N^2}\sum_{ij}\bE[G_{ij}(D_{ij}m_N)P' P^{r-2}\bar P^r]
+\frac{\O(1)}{N^3}\sum_{ij}\bE[(2-\delta_{ij})h_{ij}G_{ij}(m_N^2+P'\del_z m_N  + m_N  ) P^{r-2}\bar P^r]\\
&=\O_\prec\left(\bE\left[\frac{\Im[m_N]}{(N\eta)^2}|P'| |P|^{2r-2}\right]\right)
+\frac{\O(1)}{N^3}\sum_{p\geq1}\frac{\cC_{p+1}}{Nq^{p-1}}\sum_{ij}\bE[(2-\delta_{ij})\del_{ij}^p(G_{ij}(m_N^2+P'\del_z m_N  + m_N  ) P^{r-2}\bar P^r)]\\
&=\O_\prec\left(\bE\left[\frac{\Im[m_N]}{(N\eta)^2}|P'| |P|^{2r-2}\right]
+\max_{s\geq 2}\bE\left[\left(\frac{\Im[m_N]}{N\eta}\right)^s |P|^{2r-s}\right]\right),
\end{split}
\end{align}
where in the first term we used 
\begin{align}
\left|\frac{1}{N^2}\sum_{ij}G_{ij}(D_{ij}m_N)\right|=\left|\frac{1}{N^3}\Tr G^3\right|\leq \frac{1}{N^3}\Tr|G|^3\leq \frac{\Im[m_N]}{(N\eta)^2},
\end{align}
and we used \eqref{e:hderP} for the second term.
The same estimates hold for the second term in \eqref{e:p=1_2}.
In summary, the discussion above together implies the following bound for \eqref{e:p=1} 
\begin{align}\label{e:p=1bound}
\eqref{e:p=1}
=-\bE[m_N^2P^{r-1}\bar P^r]+\O_\prec\left(\bE\left[\frac{\Im[m_N]}{(N\eta)^2}|P'| |P|^{2r-2}\right]
+\max_{s\geq 2}\bE\left[\left(\frac{\Im[m_N]}{N\eta}\right)^s |P|^{2r-s}\right]\right).
\end{align}

For the second order term in \eqref{e:Pmoment}, i.e. $p=2$, 
\begin{align}\begin{split}\label{e:p=2}
\frac{1}{N}\sum_{ij}\frac{\cC_{3}}{Nq}\bE[\del_{ij}^{2}(G_{ij}P^{r-1}\bar P^r)]
&=\frac{1}{N}\sum_{ij}\frac{\cC_{3}}{Nq}\bE[\del_{ij}^{2}(G_{ij})P^{r-1}\bar P^r]\\
&+\frac{1}{N}\sum_{ij}\frac{2\cC_{3}}{Nq}\bE[\del_{ij}(G_{ij})\del_{ij}(P^{r-1}\bar P^r)]+\frac{1}{N}\sum_{ij}\frac{\cC_{3}}{Nq}\bE[G_{ij}\del_{ij}^{2}(P^{r-1}\bar P^r)].
\end{split}\end{align}
For the first term in \eqref{e:p=2}, thanks to Proposition \ref{p:derPandG} and \eqref{e:derivativeofG}, we have
\begin{align}
\frac{1}{N}\sum_{ij}\frac{\cC_{3}}{Nq}\bE[\del_{ij}^{2}(G_{ij})P^{r-1}\bar P^r]
=\frac{1}{N}\sum_{ij}\frac{6\cC_{3}}{Nq}\bE[G_{ij}G_{ii}G_{jj}P^{r-1}\bar P^r]
+\O_{\prec}\left(\bE\left[\frac{\Im[m_N]}{N\eta}|P|^{2r-1}\right]\right).
\end{align}
We use the identities $G_{ij}=\sum_{k\neq i}G_{ii}h_{ik}G_{kj}^{(i)}$, and $G_{kj}^{(i)}=G_{kj}-G_{ki}G_{ji}/G_{ii}$. Then by the cumulant expansion, we have
\begin{align}\begin{split}
&\phantom{{}={}}\frac{1}{N^2q}\sum_{ij}\bE\left[G_{ij}G_{ii}G_{jj}P^{r-1}\bar P^r\right]
=\frac{1}{N^2q}\sum_{ij}\bE\left[\sum_{k\neq i}h_{ik}G_{kj}^{(i)}G^2_{ii}G_{jj}P^{r-1}\bar P^r\right]\\
&=\sum_{p\geq 1}\frac{\cC_{p+1}}{N^3q^{p}}\sum_{ij}\bE\left[\sum_{k\neq i}G_{kj}^{(i)}\del_{ik}^p\left(G^2_{ii}G_{jj}P^{r-1}\bar P^r\right)\right]\\
&=\sum_{p\geq 1}\frac{\cC_{p+1}}{N^3q^{p}}\sum_{ij}\bE\left[\sum_{k\neq i}G_{kj}G^2_{ii}G_{jj}\del_{ik}^p\left(P^{r-1}\bar P^r)\right)\right]
+\O_\prec\left(\max_{s\geq 1}\bE\left[\left(\frac{\Im[m_N]}{N\eta}\right)^s|P|^{2r-s}\right]\right)\\
&=\O_\prec\left(\max_{s\geq 1}\bE\left[\frac{1}{N\eta}\left(\frac{\Im[m_N]|P'|}{N\eta}\right)^s|P|^{2r-1-s}\right]+\max_{s\geq 1}\bE\left[\left(\frac{\Im[m_N]}{N\eta}\right)^s|P|^{2r-s}\right]\right)
\end{split}\end{align} 
where we used \eqref{e:derivativeofG} in the third line and that $\|G\|_{L_2\rightarrow L_2}\leq 1/\eta$, and hence 
\begin{align}
\left|\frac{1}{N^2}\bE\left[\sum_{kj}G_{jj}G_{kj}\del_{ik}^p(P^{r-1}\bar P^r)\right]\right|
\prec \max_{s\geq 1}\bE\left[\frac{1}{N\eta}\left(\frac{\Im[m_N]|P'|}{N\eta}\right)^{s}|P|^{2r-s-1}\right].
\end{align}
For the second term in \eqref{e:p=2}, using \eqref{e:hderG}, it is given by 
\begin{align}
\begin{split}
&-\sum_{ij}\frac{2\cC_3}{N^2q}\bE[G_{ii}G_{jj}\del_{ij}(P^{r-1}\bar P^r)]+\O_{\prec}\left(\bE\left[\left(\frac{\Im[m_N]}{N\eta}\right)^2|P|^{2r-2}\right]\right).
\end{split}\end{align}
For the last term in \eqref{e:p=2}, it is given by
\begin{align}\begin{split}\label{e:p=2_3}
&\phantom{{}+{}}\frac{\O(1)}{N^2q}\sum_{ij}\bE[G_{ij}(\del_{ij}^2P)P^{r-2}\bar P^r)]+\frac{\O(1)}{N^2q}\sum_{ij}\bE[G_{ij}(\del_{ij}P)^2P^{r-3}\bar P^r)]\\
&+\frac{\O(1)}{N^2q}\sum_{ij}\bE[G_{ij} (\del_{ij}^2\bar P)^2 P^{r-1}\bar P^{r-2})]+\frac{\O(1)}{N^2q}\sum_{ij}\bE[G_{ij}(\del_{ij}^2\bar P)|P|^{2r-2}]\\
&+\frac{\O(1)}{N^2q}\sum_{ij}\bE[G_{ij}(\del_{ij}P)(\del_{ij}\bar P)P^{r-2}\bar P^{r-1}].
\end{split}\end{align}
We estimate the first and second terms; the other terms can be estimated in the same way.
Using \eqref{e:derP}, \eqref{e:hderP} and \eqref{e:derivativeofm}, we can write the first term in \eqref{e:p=2_3} as
\begin{align}\begin{split}\label{e:2Pterm}
&\phantom{{}={}}\frac{1}{N^2q}\sum_{ij}\bE  \left[G_{ij}\del_{ij}\left( P'D_{ij}m_N+(2-\delta_{ij})h_{ij}(m_N^2+m_N+P'\del_zm_N)/N \right)  P^{r-2}\bar P^r \right] \nc\\
& =\frac{1}{N^2q}\sum_{ij}\bE[G_{ij}(P'(D^2_{ij}m_N+2D_{ij}\del_zm_N\del_{ij}\tilde z)+P''(D_{ij}m_N)^2+(2-\delta_{ij})(m_N^2+m_N+P'\del_zm_N)/N)P^{r-2}\bar P^r]\\
&+\O_{\prec}\left(\bE\left[\left(\left(\frac{\Im[m_N]}{N\eta}\right)^2+\frac{\Im[m_N]|P'|}{(N\eta)^2}\right)|P|^{2r-2}\right]\right).
\end{split}\end{align}
We estimate the first term on the RHS of  \eqref{e:2Pterm} in two parts. For the first part,
\begin{align}
\begin{split}
&\phantom{{}={}}\left|\frac{1}{N^2q}\sum_{ij}G_{ij}P'(D^2_{ij}m_N)\right|
\leq \frac{\O(1)|P'|}{N^3q}\left|\sum_{ijk}G_{ij}(G_{kj}G_{ii}G_{jk}+G_{ki}G_{ji}G_{jk})\right|\\
&\leq \frac{\O(1)|P'|}{N^3q}\left|\sum_{ij}G_{ii}G_{ij}\sum_k G^2_{kj}\right|
+\frac{\O(1)|P'|}{N^3q}\sum_{ijk}|G^2_{ij}|\sum_{k}|G_{ki}G_{jk}|\\
&\prec \frac{\Im[m_N]|P'|}{q(N\eta)^2},
\end{split}\end{align}
where we used that $\|G\|_{L_2\rightarrow L_2}\leq 1/\eta$, and \eqref{e:derivativeofG}.
For the second part, we have
\begin{align}
\begin{split}
\frac{1}{N^2q}\left|\sum_{ij}G_{ij}P'D_{ij}\del_zm_N\del_{ij}\tilde z\right|
&\leq 
\frac{\O(1)}{N^4q}\left|\sum_{ij}P'G_{ij}(G^3)_{ij}(2-\delta_{ij})h_{ij}\right|\\
&\prec \frac{|P'|}{N^4q^2}\sum_{ij}\left|(G^3)_{ij}\right|
\prec \frac{\Im[m_N]|P'|}{(N\eta)^2}.
\end{split}\end{align}
For the second term on the righthand side of \eqref{e:2Pterm}, we have
\begin{align}\label{e:2MNterm}
\left|\frac{1}{N^2q}\sum_{ij}G_{ij}D_{ij}m_ND_{ij}m_N\right|
\leq \frac{1}{N^3q}\sum_{i}\left|\sum_{jk}(G_{ij}D_{ij}m_N)G_{jk}(G_{ki})\right|\prec \frac{(\Im[m_N])^2}{q(N\eta)^3},
\end{align}
where we used that $\|G\|_{L_2\rightarrow L_2}\leq 1/\eta$, $\sum_{i}|G_{ki}|^2\prec \Im[m_N]/\eta$, and $\sum_i|G_{ij}D_{ij}m_N|^2\prec (\Im[m_N])^3/N\eta^2$.
For the last term in \eqref{e:2Pterm}, since $\|G\|_{L_2\rightarrow L_2}\leq 1/\eta$, we have 
\begin{align}
\left|\frac{1}{N^2}\sum_{ij}G_{ij}\right|\leq \frac{1}{N\eta}.
\end{align}
It follows that the first term in \eqref{e:p=2_3} can be bounded by
\begin{align}\label{e:1termbound}
\O_{\prec}\left(\bE\left[\left(\left(\frac{\Im[m_N]}{N\eta}\right)^2+\frac{\Im[m_N]|P'|}{(N\eta)^2}\right)|P|^{2r-2}\right]\right).
\end{align}
The second term in \eqref{e:p=2_3} can be rewritten as 
\begin{align}\begin{split}\label{e:p=2_3_1}
&\phantom{{}={}}\frac{1}{N^2q}\sum_{ij}\bE[G_{ij}(\del_{ij}P)^2P^{r-3}\bar P^r)]
+\frac{1}{N^2q}\sum_{ij}\bE[G_{ij}(P'D_{ij}m_N)(P'D_{ij}m_N)P^{r-3}\bar P^r)]\\
&=\frac{1}{N^2q}\sum_{ij}\frac{2-\delta_{ij}}{N}\bE[h_{ij}G_{ij}(P'\del_zm_N+m_N^2+m_N)(P'D_{ij}m_N+\del_{ij}P)P^{r-3}\bar P^r)].
\end{split}
\end{align}
For the first term in \eqref{e:p=2_3_1}, we use the estimate \eqref{e:2MNterm}, and it follows that
\begin{align}
\left|\frac{1}{N^2q}\sum_{ij}\bE[G_{ij}(P'D_{ij}m_N)(P'D_{ij}m_N)P^{r-3}\bar P^r)]\right|\prec \bE\left[\frac{\Im[m_N]}{N\eta}\left(\frac{\Im[m_N]|P'|}{(N\eta)^2}\right) |P|^{2r-3} \right].
\end{align}
For the second term in \eqref{e:p=2_3_1}, we again use the cumulant expansion and find that it is given by
\begin{align}\begin{split}
&\phantom{{}={}}\frac{\cC_{p+1}}{N^3q^{p-1}}\sum_{p\geq1}\sum_{ij}\frac{2-\delta_{ij}}{N}\bE[\del_{ij}^p(G_{ij}(P'\del_zm_N+m_N^2+m_N)(P'D_{ij}m_N+\del_{ij}P)P^{r-3}\bar P^r))]\\
&=\O_{\prec}\left(\max_{s\geq 3}\bE\left[\left(\frac{\Im[m_N]}{N\eta}\right)^s|P|^{2r-s}\right]\right).
\end{split}\end{align}
It follows that the second term in \eqref{e:p=2_3} is bounded by
\begin{align}\begin{split}\label{e:2termbound}
\O_{\prec}\left( \bE\left[\frac{\Im[m_N]}{N\eta}\left(\frac{\Im[m_N]|P'|}{(N\eta)^2}\right) |P|^{2r-3} \right]+ \max_{s\geq 3}\bE\left[\left(\frac{\Im[m_N]}{N\eta}\right)^s|P|^{2r-s}\right]\right).
\end{split}\end{align}
By combining the estimates \eqref{e:1termbound} and \eqref{e:2termbound}, the last term in \eqref{e:p=2} can be bounded by
\begin{align}
\O_{\prec}\left( \max_{s=0,1}\bE\left[\left(\frac{\Im[m_N]}{N\eta}\right)^s\left(\frac{\Im[m_N]|P'|}{(N\eta)^2}\right) |P|^{2r-2-s} \right]+ \max_{s\geq 2}\bE\left[\left(\frac{\Im[m_N]}{N\eta}\right)^s|P|^{2r-s}\right]\right).
\end{align}
In summary, the discussion above implies the following bound for \eqref{e:p=2} 
\begin{align}\label{e:p=2bound}
\eqref{e:p=2}\prec \max_{s=0,1}\bE\left[\left(\frac{\Im[m_N]}{N\eta}\right)^s\left(\frac{\Im[m_N]|P'|}{(N\eta)^2}\right) |P|^{2r-2-s} \right]+ \max_{s\geq 2}\bE\left[\left(\frac{\Im[m_N]}{N\eta}\right)^s|P|^{2r-s}\right].
\end{align}


For the third order term in \eqref{e:Pmoment}, i.e., $p=3$, 
\begin{align}\label{e:p=3bound}
\begin{split}
&\phantom{{}={}}\frac{\cC_4}{N^2q^2}\sum_{ij}\bE[\del_{ij}^{3}(G_{ij}P^{r-1}\bar P^r)]
=\frac{3\cC_4}{N^2q^2}\sum_{ij}\bE[\del_{ij}(G_{ij}) \del_{ij}^2(P^{r-1}\bar P^r)]\\
&+\frac{\cC_4}{N^2q^2}\sum_{ij}\bE[\del_{ij}^3(G_{ij})P^{r-1}\bar P^r]+\O_{\prec}\left(\left(\frac{1}{q^3}+\frac{1}{N\eta}\right)\max_{s\geq 1}\bE\left[\left(\frac{\Im[m_N]|P'|}{N\eta}\right)^s |P|^{2r-1-s}\right]\right),
\end{split}\end{align}
where we used that 
$ \max_{i\neq j}|G_{ij}|\prec 1/q+1/\sqrt{N\eta}$.

For the higher order terms in \eqref{e:Pmoment}, i.e., $p\geq4$,
\begin{align}\label{e:p=4bound}
\begin{split}
\frac{\cC_{p+1}}{N^2q^{p-1}}\sum_{ij}\bE[\del_{ij}^{p}(G_{ij}P^{r-1}\bar P^r)]
&=
\frac{\cC_{p+1}}{N^2q^{p-1}}\sum_{ij}\bE[\del_{ij}^p(G_{ij})P^{r-1}\bar P^r]\\
&+\O_\prec\left(\frac{1}{q^3}\max_{s\geq 1}\bE\left[\left(\frac{\Im[m_N]|P'|}{N\eta}\right)^s |P|^{2r-1-s}\right]\right).
\end{split}\end{align}

We combine all the estimates \eqref{e:p=1bound}, \eqref{e:p=2bound}, \eqref{e:p=3bound} and \eqref{e:p=4bound} together, and use \eqref{e:hderG}
\begin{align}\begin{split}\label{e:sum1}
\bE[|P|^{2r}]
&=\bE[(Q-m_N^2)P^{r-1}\bar P^r]
+\sum_{p\geq 3}\sum_{ij}\frac{\cC_{p+1}}{Nq^{p-1}}\bE[(D_{ij}^pG_{ij})P^{r-1}\bar P^r]\\
&+\left(\bE[\cX m_N^2P^{r-1}\bar P^{r}]-\sum_{ij}\frac{2\cC_3}{N^2q}\bE[G_{ii}G_{jj}\del_{ij}(P^{r-1}\bar P^r)]-\frac{3\cC_4}{N^2q^2}\sum_{ij}\bE[G_{ii}G_{jj}\del_{ij}^{ 2} (P^{r-1}\bar P^r])\right)\\
&+\O_{\prec}\left(\left(\frac{1}{q^3}+\frac{1}{N\eta}\right)\max_{s\geq 1}\bE\left[\left(\frac{\Im[m_N]|P'|}{N\eta}\right)^s |P|^{2r-1-s}\right]+\max_{s\geq 1}\bE\left[\left(\frac{\Im[m_N]}{N\eta}\right)^s|P|^{2r-s}\right]\right).
\end{split}\end{align}
In the following we show that the second line is negligible, and then Proposition \ref{t:Pmoment} will shortly follow. We will soon see that there is a cancellation between those terms from $p=3$ and the random term $\cX m_N^2P^{r-1}\bar P^r$, which is the reason we use the random polynomial $P(z,m)$ instead of $P_0(z, m)$.   We calculate $\cX$ term:
\begin{align}
\begin{split}\label{e:p=3}
&{\phantom{{}={}}}\bE[\cX m_N^2P^{r-1}\bar P^{r}]
=\frac{1}{N}\sum_{ij}\bE[h_{ij}^2 m_N^2P^{r-1}\bar P^{r}]-\frac{1}{N^2}\sum_{ij}\bE[ m_N^2P^{r-1}\bar P^{r}]\\
&=\frac{1}{N}\sum_{ij}\sum_{p\geq 1} \frac{\cC_{p+1}}{Nq^{p-1}}\bE[\del_{ij}^p(h_{ij} m_N^2P^{r-1}\bar P^{r})]-\frac{1}{N^2}\sum_{ij}\bE[ m_N^2P^{r-1}\bar P^{r}]\\
&=\frac{1}{N}\sum_{ij} \sum_{p\geq1}\frac{\cC_{p+1}}{Nq^{p-1}}\bE[h_{ij} \del_{ij}^p(m_N^2P^{r-1}\bar P^{r})]+\frac{1}{N}\sum_{ij}\sum_{p\geq 2} \frac{p\cC_{p+1}}{Nq^{p-1}}\bE[\del_{ij}(h_{ij})\del_{ij}^{p-1}( m_N^2P^{r-1}\bar P^{r})].
\end{split}
\end{align}
For the first term in \eqref{e:p=3}, we use the cumulant expansion again,
\begin{align}\begin{split}
\sum_{ij} \sum_{p\geq1}\frac{\cC_{p+1}}{N^2q^{p-1}}\bE[h_{ij} \del_{ij}^p(m_N^2P^{r-1}\bar P^{r})]
&=\sum_{ij} \sum_{p,p'\geq1}\frac{\cC_{p+1}\cC_{p'+1}}{N^3q^{p+p'-2}}\bE[\del_{ij}^{p+p'}(m_N^2P^{r-1}\bar P^{r})]\\
&=\O_\prec\left(\max_{s\geq 2}\bE\left[\left(\frac{\Im[m_N]}{N\eta}\right)^s |P|^{2r-s}\right]\right).
\end{split}\end{align}
For the second term in \eqref{e:p=3}, we have
\begin{align}\begin{split}
&\phantom{{}={}}\sum_{ij}\sum_{p\geq 2} \frac{p\cC_{p+1}}{N^2q^{p-1}}\bE[\del_{ij}^{p-1}( m_N^2P^{r-1}\bar P^{r})]\\
&=\sum_{ij}\sum_{p\geq 2} \frac{p\cC_{p+1}}{N^2q^{p-1}}\bE[m_N^2 \del_{ij}^{p-1}( P^{r-1}\bar P^{r})]+\O_\prec\left(\max_{s\geq 2}\bE\left[\left(\frac{\Im[m_N]}{N\eta}\right)^s |P|^{2r-s}\right]\right)\\
&=\sum_{ij} \frac{2\cC_{3}}{N^2q}\bE[m_N^2 \del_{ij}( P^{r-1}\bar P^{r})]+\sum_{ij} \frac{3\cC_{4}}{N^2q^{2}}\bE[m_N^2 \del_{ij}^{2}( P^{r-1}\bar P^{r})]\\
&+\O_\prec\left(\frac{1}{q^3}\max_{s\geq 1}\bE\left[\left(\frac{\Im[m_N]|P'|}{N\eta}\right)^s |P|^{2r-1-s}\right]+\max_{s\geq 2}\bE\left[\left(\frac{\Im[m_N]}{N\eta}\right)^s |P|^{2r-s}\right]\right).
\end{split}\end{align}
Therefore, the second line in \eqref{e:sum1} simplifies to
\begin{align}
\begin{split}\label{e:cancel}
&\phantom{{}={}}\frac{2\cC_3}{N^2q}\sum_{ij}\bE[(m_N^2-G_{ii}G_{jj})\del_{ij}(P^{r-1}\bar P^r)]+\frac{3\cC_4}{N^2q^2}\sum_{ij}\bE[(m_N^2-G_{ii}G_{jj})\del_{ij}^2(P^{r-1}\bar P^r)]\\
&+\O_{\prec}\left(\frac{1}{q^3}\max_{s\geq 1}\bE\left[\left(\frac{\Im[m_N]|P'|}{N\eta}\right)^s |P|^{2r-1-s}\right]+\max_{s\geq 2}\bE\left[\left(\frac{\Im[m_N]}{N\eta}\right)^s |P|^{2r-s}\right]\right)\\
&=\frac{2\cC_3}{N^2q}\sum_{ij}\bE[(m_N^2-G_{ii}G_{jj})\del_{ij}(P^{r-1}\bar P^r)]\\
&+\O_{\prec}\left(\left(\frac{1}{q^3}+\frac{1}{N\eta}\right)\max_{s\geq 1}\bE\left[\left(\frac{\Im[m_N]|P'|}{N\eta}\right)^s |P|^{2r-1-s}\right]+\max_{s\geq 2}\bE\left[\left(\frac{\Im[m_N]}{N\eta}\right)^s |P|^{2r-s}\right]\right),
\end{split}
\end{align}
where we used that $|G_{ii}-m_N|\prec 1/q+1/\sqrt{N\eta}$.
%
For the first term on the righthand side of \eqref{e:cancel}, notice that
\begin{align}\begin{split}
m_N^2-G_{ii}G_{jj}
&=(m_N-G_{ii})m_N+(m_N-G_{jj})m_N-(m_N-G_{ii})(m_N-G_{jj})\\
&=(m_N-G_{ii})m_N+(m_N-G_{jj})m_N+\O_\prec\left(\frac{1}{q^2}+\frac{1}{N\eta}\right).
\end{split}\end{align}
Thus the  first term in the righthand side of \eqref{e:cancel} simplifies to
\begin{align}\begin{split}\label{e:cancel_1}
&\phantom{{}={}}\frac{\O(1)}{N^2q}\sum_{ij}\bE[(m_N-G_{jj})m_N(\del_{ij}P)P^{r-2}\bar P^r]
+\frac{\O(1)}{N^2q}\sum_{ij}\bE[(m_N-G_{jj})m_N(\del_{ij}\bar P)|P|^{2r-2}]\\
&+\frac{\O(1)}{N^2q}\sum_{ij}\bE[(m_N-G_{ii})m_N(\del_{ij}P)P^{r-2}\bar P^r]
+\frac{\O(1)}{N^2q}\sum_{ij}\bE[(m_N-G_{ii})m_N(\del_{ij}\bar P)|P|^{2r-2}]\\
&+ \O_\prec\left(\left(\frac{1}{q^3}+\frac{1}{N\eta}\right)\bE\left[\frac{\Im[m_N]|P'|}{N\eta}|P|^{2r-2}\right]\right).
\end{split}\end{align}
By symmetry, we only need to estimate the first term in \eqref{e:cancel_1}, as the other terms can be estimated in the same way. Using \eqref{e:derP}, we have
\begin{align}
\begin{split}\label{e:cancel_1_1}
&\phantom{{}={}}\frac{1}{N^2q}\sum_{ij}\bE[(m_N-G_{jj})m_N(\del_{ij}P)P^{r-2}\bar P^r]
=\frac{1}{N^2q}\sum_{ij}\bE[(m_N-G_{jj})m_NP'(D_{ij}m_N)P^{r-2}\bar P^r]\\
&+\frac{1}{N^2q}\sum_{ij}\frac{2-\delta_{ij}}{N}\bE[h_{ij}(m_N-G_{jj})m_N(P'\del_zm_N+m_N^2+m_N)P^{r-2}\bar P^r].
\end{split}
\end{align}
We have that $D_{ij}m_N=\sum_k G_{ki}G_{kj}/N$. Using the identity $G_{ki}=\sum_{\ell\neq i}h_{i\ell}G_{\ell k}^{(i)}G_{ii}$, we can write the first term in \eqref{e:cancel_1_1} as
\begin{align}\begin{split}\label{e:cancel_1_1_1}
&\phantom{{}={}}\frac{1}{N^3q}\sum_{ijk}\bE[G_{ki}G_{kj}(m_N-G_{jj})m_NP'P^{r-2}\bar P^r]\\
&=\frac{1}{N^3q}\sum_{ijk}\sum_{\ell\neq i}\bE[h_{i\ell}G^{(i)}_{\ell k}G_{ii}G_{kj}(m_N-G_{jj})m_NP'P^{r-2}\bar P^r]\\
&=\sum_{p\geq 1}\frac{\cC_{p+1}}{N^4q^{p}}\sum_{ijk}\sum_{\ell\neq i}\bE[G^{(i)}_{\ell k}\del_{i\ell}^p(G_{ii}G_{kj}(m_N-G_{jj})m_NP'P^{r-2}\bar P^r)]\\
&=\frac{\O(1)}{N^4q}\sum_{ijk}\sum_{\ell\neq i}\bE[G^{(i)}_{\ell k}G_{ii}G_{kj}(m_N-G_{jj})m_NP'(\del_{i\ell}P)P^{r-3}\bar P^r)]\\
&+\frac{\O(1)}{N^4q}\sum_{ijk}\sum_{\ell\neq i}\bE[G^{(i)}_{\ell k}G_{ii}G_{kj}(m_N-G_{jj})m_NP'(\del_{i\ell}\bar P)P^{r-2}\bar P^{r-1})]\\
&+\O_\prec\left(\max_{s\geq 1}\bE\left[\left(\frac{1}{q^3}+\frac{1}{N\eta}\right)\left(\frac{\Im[m_N]|P'|}{N\eta}\right)^s|P|^{2r-1-s}\right]\right).
\end{split}\end{align}
For the first term on the righthand side of \eqref{e:cancel_1_1_1}, we have
\begin{align}\begin{split}
\left|\frac{1}{N^4q}\sum_{ijk}\sum_{\ell\neq i}G^{(i)}_{\ell k}G_{ii}G_{kj}(m_N-G_{jj})(\del_{i\ell}P)\right|
&\leq \frac{1}{N^4q}\sum_{ij}|G_{ii}|\left|\sum_{\ell\neq i, k}(\del_{i\ell}P)G^{(i)}_{\ell k}(G_{kj}(m_N-G_{jj}))\right|\\
&\prec \frac{\Im[m_N]|P'|}{(N\eta)^2}\sqrt{\frac{\Im[m_N]}{N\eta}}\left(\frac{1}{q^2}+\frac{1}{N\eta}\right),
\end{split}\end{align}
where we used that $\|G^{(i)}\|_{L_2\rightarrow L_2}\leq 1/\eta$, $|\del_{i\ell}P|\prec \Im[m_N]|P'|/N\eta$, $|m_N-G_{jj}|\prec 1/q+1/\sqrt{N\eta}$, and $\sum_{k}|G_{kj}|^2\leq \Im[m_N]/\eta$.
The same estimate holds for the second term on the righthand side of \eqref{e:cancel_1_1_1}. Combining the above estimates, it follows that the first term in \eqref{e:cancel_1_1} is bounded by
\begin{align}\label{e:cancelline2}
\O_{\prec}\left(\max_{s\geq 1}\bE\left[\left(\frac{1}{q^3}+\frac{1}{N\eta}\right)\left(\frac{\Im[m_N]|P'|}{N\eta}\right)^s|P|^{2r-1-s}\right]+\bE\left[\frac{\Im[m_N]|P'|}{(N\eta)^2}\sqrt{\frac{\Im[m_N]|P'|}{N\eta}}\left(\frac{1}{q^2}+\frac{1}{N\eta}\right)|P|^{2r-3}\right]\right).
\end{align}
For the second term on righthand side of \eqref{e:cancel_1_1}, we use the cumulant expansion again, 
\begin{align}\begin{split}\label{e:cancelline3}
&\phantom{{}={}}\frac{1}{N^2q}\sum_{ij}\frac{2-\delta_{ij}}{N}\bE[h_{ij}(m_N-G_{jj})m_N(P'\del_zm_N+m_N^2+m_N)P^{r-2}\bar P^r]\\
&=\sum_{p\geq 1}\frac{\cC_{p+1}}{N^3q^p}\sum_{ij}\frac{2-\delta_{ij}}{N}\bE[\del_{ij}^p((m_N-G_{jj})m_N(P'\del_zm_N+m_N^2+m_N)P^{r-2}\bar P^r)]\\
&=\O_\prec\left(\max_{s\geq2}\bE\left[\left(\frac{\Im[m_N]}{N\eta}\right)^s|P|^{2r-s}\right]\right).
\end{split}\end{align}

The estimates \eqref{e:sum1}, \eqref{e:cancel}, \eqref{e:cancel_1}, \eqref{e:cancel_1_1}, \eqref{e:cancelline2} and \eqref{e:cancelline3} all together lead to the following estimate 
\begin{align}\begin{split}\label{e:sum2}
\bE[|P|^{2r}]
&=\bE[(Q-m_N^2)P^{r-1}\bar P^r]
+\sum_{p\geq 3}\sum_{ij}\frac{\cC_{p+1}}{Nq^{p-1}}\bE[(D_{ij}^pG_{ij})P^{r-1}\bar P^r]\\
&+\O_\prec\left(\max_{s_1+s_2\geq1}\bE\left[\left(\left(\frac{1}{q^3}+\frac{1}{N\eta}\right)\left(\frac{\Im[m_N]|P'|}{N\eta}\right)\right)^{s_1/2}\left(\frac{\Im[m_N]}{N\eta}\right)^{s_2}|P|^{2r-s_1-s_2}\right]\right).
\end{split}\end{align}
We have left the first line in \eqref{e:sum2}.  This is estimated similarly to Proposition \ref{p:DSE}. 
More precisely, we repeat the procedure as in Proposition \ref{p:DSE}, by repeatedly using the cumulant expansion. Derivatives hitting $\tilde z$ give an extra copy of $\Im[m_N]/N\eta$ (see \eqref{e:hderG}) and so the resulting term can be bounded by
\begin{align}
\O_\prec \left(\max_{s\geq 1}\bE\left[\left(\frac{\Im[m_N]}{N\eta}\right)^s|P|^{2r-s}\right]\right).
\end{align}
If the derivative hits $P$ or $\bar P$, we get a factor of $\Im[m_N]|P'|/qN\eta$, and the resulting term can be bounded by
\begin{align}
\O_\prec \left(\max_{s\geq 0}\bE\left[\frac{\Im[m_N]|P'|}{q^3N\eta}\left(\frac{\Im[m_N]}{N\eta}\right)^s|P|^{2r-s-2}\right]\right).
\end{align}
If the derivative never hits $\tilde z,P, \bar P$, these terms will cancel with $Q$. The estimate \eqref{e:Pmoment} follows.



\end{proof}

%
%

\subsection{Proof of Theorem \ref{thm:edgerigidity} }
\label{s:proofrigidity}
We only analyze the behavior of the Stieltjes transform $m_N$ close to the right edge $\tilde L$. The case that $m_N$ close to the left edge can be analyzed in the same way.  Recall the shifted spectral domain $\cD$ from \eqref{e:defD}.

\begin{proposition}\label{p:stable}
There exists a constant $\varepsilon>0$ such that the following holds. Suppose that $\delta:\cD\rightarrow \mathbb{R}$ is a function so that 
\begin{align}
|P(L+\cX+z, m_N(L+\cX+z))|\leq \delta(z).
\end{align}
Suppose that $N^{-2}\leq \delta(z)\leq \varepsilon$ for $z\in \cD$, that $\delta$ is Lipschitz continuous with Lipschitz constant $N$ and moreover that for each fixed $\kappa$ the function $\eta\mapsto \delta(\kappa+\ri\eta)$ is nonincreasing for $\eta>0$.   Then,
\begin{align}
|m_N(L+\cX+z)-\tilde m_\infty(L+\cX+z)|=\OO\left(\frac{\delta(z)}{\sqrt{|\kappa|+\eta+\delta(z)}}\right),
\end{align}
where the implicit constant is independent of $z$ and $N$.
\end{proposition}
\begin{proof}
We abbreviate $\tilde z=L+\cX+z$. 
 From \eqref{e:derPest}, we have $|\del_2 P(\tilde z, \tilde m_{\infty}(\tilde z))|\asymp \sqrt{|\kappa|+\eta}$, and $\del_2^2 P(\tilde z, \tilde m_{\infty}(\tilde z))=1+\OO(1/q^2)$. By  a Taylor expansion, we have
\begin{align}
-P(\tilde z,  m_N(\tilde z))+\del_2 P(\tilde z, \tilde m_\infty(\tilde z))(m_N(\tilde z)-\tilde m_\infty(\tilde z))+(1+\oo(1))(m_N(\tilde z)-\tilde m_\infty(\tilde z))^2=0,
\end{align}
We abbreviate $\cR(\tilde z)\deq P(\tilde z,m_N(\tilde z))$. There exists $a(\tilde z)\asymp\sqrt{|\kappa|+\eta}$ and $b(\tilde z)\asymp 1$, such that
\begin{align}\label{e:cReq}
\cR(\tilde z)=a(\tilde z)(m_N(\tilde z)-\tilde m_\infty(\tilde z))+b(\tilde z)(m_N(\tilde z)-\tilde m_\infty(\tilde z))^2.
\end{align}
With \eqref{e:cReq} in hand, Proposition \ref{p:stable} follows by a continuity argument essentially the same as \cite[Lemma 4.5]{MR3183577}.
\end{proof}


Before proving Theorem \ref{thm:edgerigidity}, we first prove a weaker estimate.
\begin{proposition}
Let $H$ be as in Definition \ref{asup}. Let $m_N(z)$ be the Stieltjes transform of its eigenvalue density, and $\tilde m_\infty(z)$ as defined in Proposition \ref{p:tminfty}. Uniformly for any $z=\kappa+\ri\eta\in \cD$ as defined in \eqref{e:recalldefD}, letting $\tilde z=L+\cX+z$, we have
\begin{align}
|m_N(\tilde z)-\tilde m_\infty(\tilde z)|\prec \sqrt{|\kappa|+\eta}.
\end{align}
\end{proposition} 
\begin{proof}
For $|\kappa|+\eta\gg 1/q^3\sqrt{N}$, and $\kappa\leq 1$, by Proposition \ref{p:tminfty}, we have
\begin{align}
\Im[\tilde m_{\infty}(\tilde z)]\asymp \Phi(z)\deq \left\{
\begin{array}{cc}
\sqrt{|\kappa|+\eta}, & \kappa\leq 0,\\
\eta/\sqrt{|\kappa|+\eta}, & \kappa\geq 0.
\end{array}
\right.
\end{align}
and 
\begin{align}
|\del_2 P(\tilde z, \tilde m_{\infty}(\tilde z))|\asymp \sqrt{|\kappa|+\eta}.
\end{align}
We denote 
\begin{align}
\Lambda_N(z)\deq |m_N(L+\cX+z)-\tilde m_{\infty}(L+\cX+z)|.
\end{align}
Then we have
\begin{align}
\Im[m_N(\tilde z)]\lesssim \Phi(z)+\Lambda_N(z),
\end{align}
and by Proposition \ref{p:tminfty}
\begin{align}
\del_2 P(\tilde z, m_N(\tilde z))
=\del_2 P(\tilde z, \tilde m_\infty(\tilde z))+\O(|m_N(\tilde z)-\tilde m_\infty(\tilde z)|)
=\O(\sqrt{|\kappa|+\eta}+\Lambda_N(z)).
\end{align}
By H{\" o}lder's inequality we obtain from Proposition \ref{t:Pmoment}, 
\begin{align}\label{e:momentbound}
\bE[|P(\tilde z, m_N(\tilde z))|^{2r}]
\prec\frac{1}{(N\eta)^r}\left(\frac{1}{q^3}+\frac{1}{N\eta}\right)^r\bE\left[\Lambda_N(z)^{2r}+(|\kappa|+\eta)^{r/2}\left(\Phi(z)^r+\Lambda_N(z)^r\right)\right].
\end{align}
With overwhelming probability we have the following Taylor expansion,
\begin{align}
\nonumber P(\tilde z, m_N(\tilde z))
&=P(\tilde z, \tilde m_\infty(\tilde z))
+\del_2 P(\tilde z, \tilde m_\infty(\tilde z))(m_N(\tilde z)-\tilde m_\infty(\tilde z))+\frac{\del_2^2 P(\tilde z, \tilde m_\infty(\tilde z))+\oo(1)}{2}(m_N(\tilde z)-\tilde m_\infty(\tilde z))^2\\
&=\del_2 P(\tilde z, \tilde m_\infty(\tilde z))(m_N(\tilde z)-\tilde m_\infty(\tilde z))+(1+\oo(1))(m_N(\tilde z)-\tilde m_\infty(\tilde z))^2,\label{e:taylorexpand}\end{align}
where we used that $\del_2^2 P(\tilde z, \tilde m_\infty(\tilde z))=2+\O(1/q^2)$ and $\Lambda_N(z)\ll1$ with overwhelming probability.   Rearranging the last equation and using the definition of $\Lambda_N(z)$, we have arrived at 
\begin{align}\label{e:Labda2bound}
\Lambda_N(z)^2\lesssim\Lambda_N(z)\sqrt{|\kappa|+\eta}+ |P(\tilde z,m_N(\tilde z))|, 
\end{align}
and thus 
\begin{align}\label{e:Lambda4r}
\bE[\Lambda_N(z)^{4r}]\lesssim(|\kappa|+\eta)^r\bE[\Lambda_N(z)^{2r}]+ \bE[|P(\tilde z,m_N(\tilde z))|^{2r}].
\end{align}
We replace $\bE[|P(\tilde z,m_N(\tilde z))|^{2r}]$ in \eqref{e:Lambda4r} by \eqref{e:momentbound}. Moreover, on the domain $\cD$, we have $N\eta\sqrt{|\kappa|+\eta}\gg1$ and $N\eta(|\kappa|+\eta)q^3\gg1$, and so,
\begin{align}
\bE[\Lambda_N(z)^{4r}]\prec (|\kappa|+\eta)^{2r},
\end{align}
and by Markov's inequality we get $\Lambda_N(z)\prec \sqrt{|\kappa|+\eta}$.

\end{proof}

\begin{proof}[Proof of Theorem \ref{thm:edgerigidity}]

We assume that there exists some deterministic control parameter $\Lambda(z)$ such that the prior estimate holds
\begin{align}
|m_N(L+\cX+z)-\tilde m_{\infty}(L+\cX+z)|\prec \Lambda(z)\lesssim \sqrt{|\kappa|+\eta}.
\end{align}
%
Since $\Phi(z)\gtrsim\sqrt{|\kappa|+\eta}$ and $\Lambda(z)\prec \sqrt{|\kappa|+\eta}$, \eqref{e:momentbound}  simplifies to
\begin{align}\label{e:estimateP}
|P(\tilde z, m_N(\tilde z))|
&\prec \left(\frac{1}{q^{3/2}\sqrt{N\eta}}+\frac{1}{N\eta}\right)\left((\Lambda(z)+\Phi(z))\sqrt{|\kappa|+\eta}\right)^{1/2}.
\end{align}

If $\kappa\geq 0$, then $\Phi(z)=\eta/ \sqrt{|\kappa|+\eta}$, and \eqref{e:estimateP} simplifies to
\begin{align}\label{e:boundoutside}
|P(\tilde z, m_N(\tilde z))|
\prec \frac{1}{N\eta^{1/2}}+\frac{1}{N^{1/2}q^{3/2}}+\frac{(|\kappa|+\eta)^{1/4}\Lambda^{1/2}}{N\eta}+\frac{(|\kappa|+\eta)^{1/4}\Lambda^{1/2}}{q^{3/2}(N\eta)^{1/2}}.
\end{align}
Thanks to Proposition \ref{p:stable}, by taking $\delta(z)$ the righthand side of \eqref{e:boundoutside} times $N^{\fc}$, we have
\begin{align}\begin{split}\label{e:outS0}
|m_N(\tilde z)-\tilde m_\infty(\tilde z)|\prec \frac{1}{\sqrt{|\kappa|+\eta}}\left( \frac{1}{N\eta^{1/2}}+\frac{1}{N^{1/2}q^{3/2}}+\frac{(|\kappa|+\eta)^{1/4}\Lambda^{1/2}}{N\eta}
+\frac{(|\kappa|+\eta)^{1/4}\Lambda^{1/2}}{q^{3/2}(N\eta)^{1/2}}\right).
\end{split}\end{align}
By iterating \eqref{e:outS0}, we get
\begin{align}\label{e:outS}
|m_N(\tilde z)-\tilde m_\infty(\tilde z)|\prec \frac{1}{\sqrt{|\kappa|+\eta}}\left(\frac{1}{N\eta^{1/2}}+\frac{1}{N^{1/2}q^{3/2}}+\frac{1}{(N\eta)^2}+\frac{1}{q^3 N\eta}\right).
\end{align}
This finishes the proof of \eqref{e:stateoutS}.

If $\kappa\leq 0$, then $\Phi(z)=\sqrt{|\kappa|+\eta}$ and $\Lambda(z)\prec \sqrt{|\kappa|+\eta}$, \eqref{e:estimateP} simplifies to
\begin{align}\label{e:boundinside}
|P(\tilde z,m_N(\tilde z)|
\prec \frac{(|\kappa|+\eta)^{1/2}}{N\eta}+\frac{(|\kappa|+\eta)^{1/2}}{(N\eta)^{1/2}q^{3/2}}.
\end{align}
It follows from Proposition \ref{p:stable}, by taking $\delta(z)$ the righthand side of \eqref{e:boundinside} times $N^\fc$, we have
\begin{align}\label{e:inS}
|m_N(\tilde z)-\tilde m_\infty(\tilde z)|\prec \frac{1}{N\eta}+\frac{1}{(N\eta)^{1/2}q^{3/2}}.
\end{align}
This finishes the proof of \eqref{e:stateinS}.
\end{proof}

\section{Edge statistics of $H(t)$} \label{sec:ht}
Let $H$ be as in Definition \ref{asup}. In this section we assume that $q\gg N^{1/9}$, and consider the ensemble
\beq\label{e:Ht}
H(t) := \e^{ - t/2}H + \left( 1 - \e^{ -t} \right)^{1/2} W,
\eeq
where $H(0)=H$ and $W$ is an independent GOE matrix.  The random matrix $H(t)$ also satisfies the properties of Definition \ref{asup}. Thanks to Proposition \ref{p:minfty} and \ref{p:tminfty}, we can construct a random probability measure for $H(t)$, which is supported on $[-\tilde L_t, \tilde L_t]$, where
\begin{align}\label{e:defLt}
\tilde L_t =L_t+\cX_t+\O_{\prec}\left(\frac{1}{\sqrt{N}q^3}\right), \quad \cX_t=\frac{1}{N}\sum_{ij}\left(h_{ij}^2(t)-\frac{1}{N}\right).
\end{align}
%

We take $t\gg N^{-1/3}$ as
\beq\label{e:deft}
t = \frac{N^{\fd}}{N^{1/3}},
\eeq
for a $1/3 > \fd >0$ to be determined.
We denote $\cA$ to be the set of sparse random matrices $H$, such that \eqref{e:stateoutS} and \eqref{e:stateinS} hold at edges $\pm \tilde L$:
\begin{align}\label{e:defA}
\cA\deq \{H: \text{\eqref{e:stateoutS} and \eqref{e:stateinS} hold around edges $\pm (L+\cX)$} \},
\end{align}
By Theorem \ref{thm:edgerigidity}, we know that the event holds with $\cA$ holds with probability $\bP(\cA)\geq 1-N^{-D}$ for any $D\geq 0$. 
We denote the empirical eigenvalue distribution of $H$ by $\mu_{H}$. We denote the eigenvalues of $H(t)$ as $\la_1(t), \la_2(t),\cdots, \la_N(t)$.

In this section, we prove the following theorem, which states that the fluctuations of extreme eigenvalues of $H(t)$ are given by a combination of the Tracy-Widom distribution and the Gaussian distribution. 
\begin{theorem} \label{thm:htedge}
Let $H$ be as in Definition \ref{asup} with $q\gg N^{1/9}$. Let $H(t)$ be as in \eqref{e:Ht}, with eigenvalues denoted by $\la_1(t),\la_2(t),\cdots,\lambda_N(t)$, and $t=N^{-1/3+\fd}$.  Let $k\geq 1$ and $F : \rr^k \to \rr$ be a bounded test function with bounded derivatives.  There is a universal constant $\fc>0$ depending on $\fd$, so that for any  $H\in \cA$ as defined in \eqref{e:defA} we have
\begin{align}\begin{split}\label{eqn:htedge2}
&\phantom{{}={}} \ee_{W}[ F (N^{2/3} ( \lambda_1(t) - L_t - \cX ), \cdots , N^{2/3} ( \lambda_k(t) - L_t - \cX ) | H] \\
&= \ee_{GOE}[ F (N^{2/3} ( \mu_1 - 2), \cdots, N^{2/3} ( \mu_k - 2  ) ) +\O\left(N^{-\fc}\right),
\end{split}\end{align}
where the expectation on the righthand side is with respect to a GOE matrix with eigenvalues denoted by $\mu_i$. 
Moreover, one can also put the $\cX$ on the righthand side of \eqref{eqn:htedge2},
\begin{align}\begin{split} \label{eqn:htedge1}
&\phantom{{}={}}\ee_{H, W}[ F (N^{2/3} ( \lambda_1(t) - L_t ), \cdots , N^{2/3} ( \lambda_k (t)- L_t ) ]\\
& = \ee_{H,GOE}[ F (N^{2/3} ( \mu_1 - 2 + \cX), \cdots, N^{2/3} ( \mu_k - 2 + \cX ) ) ]+\O\left( N^{-\fc}\right),
\end{split}\end{align}
where expectation on the righthand is with respect to a GOE matrix with eigenvalues $\mu_i$, and a sparse random matrix $H$ (note that on the RHS, $H$ is independent of the GOE and only enters the expecation through $\cX$).
\end{theorem}

Take $\eta_*=N^{-2/3+\fd/2}$, and $\tilde z=L+\cX+z$, where $z=\kappa+\ri\eta$. For any $H\in \cA$, from the defining relations of $\cA$, i.e. \eqref{e:stateoutS} and \eqref{e:stateinS}, we have 
\begin{align}
|m_N(\tilde z)-\tilde m_{\infty}(\tilde z)|\ll\frac{\eta}{\sqrt{\kappa+\eta}},
\end{align}
for $0\leq \kappa\leq 1$ and $\eta_*\leq \eta\leq \fb$, and 
\begin{align}
|m_N(\tilde z)-\tilde m_{\infty}(\tilde z)|\ll\sqrt{|\kappa|+\eta},
\end{align}
for $-1\leq \kappa\leq 0$ and $\eta_*\leq \eta\leq \fb$.
Hence, combining with estimates \eqref{eqn:imtilminf}, we get
\begin{align}
\Im [ m_N ( \tilde z) ] \asymp \Im [ \tilde{m}_\infty(\tilde z)], \quad \eta_*\leq \eta \leq \fb,\quad |\kappa|\leq 1,
\end{align}
for $H\in \cA$. Moreover, in the regime $q\gg N^{-1/9}$, \eqref{e:stateoutS} also implies \eqref{e:largeeig} such that $\lambda_1 (0) - \tilde{L}  \leq N^{-2/3+\fd/2}$.   Hence, $H$ is $\eta*$-regular in the sense of \cite[Definition 2.1]{LY}, and the result of \cite[Theorem 2.2]{LY} applies for $t=N^{-1/3+\fd}$ as above. This result gives the limiting distribution of the extreme eigenvalues of $H_t$.  The result involves scaling parameters coming from the free convolution, which we must now define.

We denote $\rhosc(x)$ the semicircle law which is the limit eigenvalue density of a Gaussian orthogonal ensemble $W$. The semicircle law at time $t$, i.e. the limit eigenvalue density of $(1-e^{-t})^{1/2}W$, is given by $(1-e^{-t})^{-1/2}\rhosc((1-e^{-t})^{-1/2}x)\rd x$. We denote the free convolution of the empirical eigenvalue density of $e^{-t/2}H$, with the semicircle law at time $t$ by $\hatmfc$, and the free convolution of $e^{-t/2}\tilde \rho_\infty(e^{t/2}x)$ (as defined in Proposition \ref{p:tminfty}) by $\tilmfc$.  They satisfy the functional equations
\beq
\tilmfc(z) = \int \frac{ \tilrhoinf (x)\rd x}{ \e^{-t/2} x - \tilxi (z)}, \qquad \hatmfc(z) = \int \frac{ \d \mu_H (x) }{ \e^{-t/2} x - \hatxi(z) },
\eeq
where
\beq
\tilxi (z) := z + (1-\e^{-t} ) \tilmfc (z), \qquad \hatxi (z):= z + (1-\e^{-t} ) \hatmfc (z).
\eeq
For $t=N^{-1/3+\fd}$, these measures have densities which are supported on a single interval, with square root behaviour at the edges.  The edges are defined as follows.  Let $\tilxip$ and $\hatxip$ be the largest real solutions to 
\beq\label{e:defxi}
1 = (1 - \e^{-t} ) \int \frac{  \tilrhoinf (x)\rd x } { ( \e^{-t/2} x - \tilxip )^2 },\quad 
1 = (1 - \e^{-t} ) \int \frac{  \rd \mu_{H}(x) } { ( \e^{-t/2} x - \hatxip )^2 },
\eeq
Then the edges $\tilEp$ and $\hatEp$ of the free convolution with Stieltjes transforms $\tilmfc(z)$ and $\hatmfc(z)$ are defined by $\tilxip = \tilxi ( \tilEp)$, $\hatxip=\hatxi(\hatEp)$ respectively.   We introduce the scaling parameters $\hatg$ and $\tilg$ as follows.
\beq\tilg := \left(  -(1-\e^{-t} )^3 \int \frac{ \tilrhoinf  (x) \d x}{ (\e^{-t/2} - \tilxip )^3 } \right)^{-1/3},\quad 
\hatg := \left(  -(1-\e^{-t} )^3 \int \frac{ \d \mu_H (x)}{ (\e^{-t/2} - \hatxip )^3 } \right)^{-1/3}.
\eeq
The main result of \cite[Theorem 2.2]{LY} states that for any $\eta_*$-regular $H$, and smooth test function $F : \rr^k \mapsto \rr$, there exists a universal constant $\fc>0$ depending only on $\fd >0$ as above, such that 
\begin{align}\begin{split}\label{e:htedge3}
&\phantom{{}={}} \ee_W[ F ( N^{2/3} \hatgam ( \lambda_1(t) - \hatEp ), \cdots N^{2/3} \hatgam ( \lambda_k(t) - \hatEp ) ) | H ] \\
&= \ee_{GOE} [ F (N^{2/3} ( \mu_1 - 2 ), \cdots N^{2/3} ( \mu_k - 2 ) ) ]+\O\left( N^{-\fc}\right),
\end{split}\end{align}
where the expectation on the righthand side is with respect to a Gaussian orthogonal ensemble with largest few eigenvalues $\mu_1, \mu_2,\cdots,\mu_k$.
 The parameters $\hatgam$ and $\hatEp$ depend on $H$.  The following two propositions show that they are close to $\tilg$ and $\tilEp$, which we can calculate explicitly.  The proofs are deferred to Appendix \ref{a:fc}.

\begin{proposition} \label{lem:freeconv1}
Under the assumptions of Theorem \ref{thm:htedge}, we have,
\beq
\e^{t/2} \hatxip - \tilL \asymp t^2, \qquad \e^{t/2} \tilxip - \tilL \asymp t^2,
\eeq
and the estimate
\beq \label{eqn:xiest}
| \hatxip- \tilxip | \prec t^2 \left( \frac{1}{N^{1/2}  q^{3/2}t^2 } + \frac{1}{ N t^3} \right).
\eeq
For the edges of the free convolution law we have
\beq \label{eqn:edgeest}
| \hatEp - \tilEp | \prec \frac{1}{ N^{1/2} q^{3/2} } + \frac{1}{ N t}.
\eeq
For the scaling parameter we have
\beq \label{eqn:scalest}
|\hatgam-\tilg|=(1-\e^{-t} )^3\e^{3t/2} \left|  \int \frac{ \d \mu_H(x) } { (x - \e^{t/2} \hatxip )^3} -  \int \frac{ \tilrhoinf (x)\rd x }{ (x - \e^{t/2} \tilxip )^3 } \right| \prec \frac{1}{ N t^3} + \frac{1}{ N^{1/2}  q^{3/2} t^2}.
\eeq
\end{proposition}


\begin{proposition} \label{lem:freeconv2}
Under the assumptions of Theorem \ref{thm:htedge}, there exists a universal constant $\fc>0$, such that
\begin{align}\begin{split}\label{e:tEestimate}
\tilEp &= 2 + \cX + \frac{ 6 \mathcal{C}_4}{q^2}(1 - 2 t)  + \frac{ 120 \mathcal{C}_6}{q^4} - \frac{81 \mathcal{C}_4^2}{q^4}+\O_\prec \left(N^{-2/3-\fc}\right). \\
&= L_t + \cX+\O_\prec \left(N^{-2/3-\fc}\right).  \\
&= \tilde{L}_t +\O_\prec \left(N^{-2/3-\fc}\right). 
\end{split}\end{align}
and
\beq\label{e:tgestimate}
\tilg = 1 + \O_\prec \left( N^{-\fc} \right).
\eeq
\end{proposition}

Thanks to Propositions \ref{lem:freeconv1} and \ref{lem:freeconv2}, we can replace $\hatgam$ and $\hatEp$ in \eqref{e:htedge3} by $L_t+\cX$ and $1$ respectively, which gives an error of size $\O(N^{-\fc})$.
Thus, we have
\begin{align}\begin{split}
&\phantom{{}={}} \ee_W[ F ( N^{2/3} ( \lambda_1(t) - L_t - \cX ), \cdots N^{2/3}  ( \lambda_k(t) - L_t - \cX ) ) | H ]\\
& =\ee_{GOE} [ F (N^{2/3} ( \mu_1 - 2 ), \cdots N^{2/3} ( \mu_k - 2 ) ) ] +\O\left(N^{-\fc}\right).
\end{split}\end{align}
This finishes the proof of \eqref{eqn:htedge2}.

In \eqref{eqn:htedge2}, we can take $F(x) = F_1 (x + N^{2/3} \cX)$ for any smooth function $F_1$, as the estimates only depend on $||F'||_\infty$ and $||F||_\infty$.  Taking expectation over $H$ we then see that
\begin{align}\begin{split}
&\phantom{{}={}}\ee_{H, W}[ F (N^{2/3} ( \lambda_1(t) - L_t ), \cdots , N^{2/3} ( \lambda_k (t)- L_t ) ]\\
& = \ee_{H,GOE}[ F (N^{2/3} ( \mu_1 - 2 + \cX), \cdots, N^{2/3} ( \mu_k - 2 + \cX ) ) ]+\O\left( N^{-\fc}\right).
\end{split}\end{align}
This yields Theorem \ref{thm:htedge}.

\section{Comparison: Proof of Theorem \ref{thm:Tracy-Widom}}\label{s:comparison}

We recall $H(t)$ from \eqref{e:Ht}. For simplicity of notation, in this section we denote the Stieltjes transform of the eigenvalue density of $H(t)$ as $m_t(z)$, and the Stieltjes transform of the eigenvalue density of $H$ as $m(z)$. 
In this section we prove the following theorem, which states that for $t=N^{-1/3+\fd}$, the rescaled extreme eigenvalues of $H$ and $H(t)$  have the same distribution. 
\begin{theorem}\label{thm:comp}
Let $H$ be as in Definition \ref{asup} with $q\gg N^{1/9}$. Let $H(t)$ be as in \eqref{e:Ht}, with eigenvalues denoted by $\la_1(t),\la_2(t),\cdots,\lambda_N(t)$, and $t=N^{-1/3+\fd}$.  Fix $k\geq 1$ and numbers $s_1,s_2,\cdots, s_k$, there is a universal constant $\fc>0$ so that,
\begin{align}\begin{split}\label{e:comp1}
&\phantom{{}={}} \bP_{H}\left( N^{2/3} ( \lambda_i(0) - L - \cX )\geq s_i,1\leq i\leq k \right)\\
&= \bP_{H(t)}\left( N^{2/3} ( \lambda_i(t) - L_t - \cX_t )\geq s_i,1\leq i\leq k \right) +\O\left(N^{-\fc}\right),
\end{split}\end{align}
where $L_t,\cX_t$ are as defined in \eqref{e:defLt}. In the special case, $q=C N^{1/6}$, we have
\begin{align}\begin{split}\label{e:comp2}
&\phantom{{}={}} \bP_{H}\left( N^{2/3} ( \lambda_i(0) - L )\geq s_i,1\leq i\leq k \right)\\
&= \bP_{H(t)}\left( N^{2/3} ( \lambda_i(t) - L_t)\geq s_i,1\leq i\leq k \right) +\O\left(N^{-\fc}\right).
\end{split}\end{align}
The analogous statement holds for the smallest eigenvalues.
\end{theorem}

Theorem \ref{thm:Tracy-Widom} follows from combining Theorem \ref{thm:htedge} and Theorem \ref{thm:comp}.
\begin{proof}[Proof of Theorem \ref{thm:Tracy-Widom}]  The result 
\eqref{eqn:htedgeb1} follows from combining \eqref{eqn:htedge2} and \eqref{e:comp1}. The result 
\eqref{eqn:htedgeb2} follows from combining \eqref{eqn:htedge1} and \eqref{e:comp2}.
\end{proof}

Before proving Theorem \ref{thm:comp}, we need to introduce some notation. For any $E$, we define 
\begin{align}
\cN_t(E)\deq |\{i:\la_i(t)\geq L_t+\cX+E\}|,
\end{align}
and we write $\cN_0(E)$ as $\cN(E)$. We fix $\fc>0$, and take $\ell=N^{-2/3-\fc/3}$ and $\eta=N^{-2/3-\fc}$. Then with overwhelming probability, from \eqref{e:largeeig}, we know that $\la_1(t)\leq L_t+\cX_t+N^{-2/3+\fc}$. We define:
\begin{align}
\chi_E(x)={\bf 1}_{[E, N^{-2/3+\fc}]}(x-L_t-\cX_t),\quad \theta_\eta(x)\deq \frac{\eta}{\pi(x^2+\eta^2)}=\frac{1}{\pi}\Im\frac{1}{x+\ri \eta}.
\end{align}
From the same argument as in \cite[Lemma 2.7]{MR3034787}, we get that 
\begin{align}\label{e:resolventexp}
\Tr(\chi_{E+\ell}*\theta_\eta)(H(t))-N^{-\fc/9}\leq \cN_t(E)\leq \Tr(\chi_{E-\ell}*\theta_\eta)(H(t))+N^{-\fc/9},
\end{align}
hold with overwhelming probability. Let $K_i: \bR\mapsto[0,1]$ be a monotonic smooth function satisfying,
\begin{align}
K_i(x)=\left\{\begin{array}{cc}
0 & x\leq i-2/3,\\
1& x\geq i-1/3.
\end{array}\right.
\end{align}
We have that ${\bf 1}_{\cN_t(E)\geq i}=K_i(\cN_t(E))$, and since $K_i$ is monotonically increasing, and so 
\begin{align}
K_i\left(\Tr(\chi_{E+\ell}*\theta_\eta)(H(t))\right)+\O(N^{-\fc/9})\leq {\bf 1}_{\cN_t(E)\geq i}\leq K_i\left(\Tr(\chi_{E-\ell}*\theta_\eta)(H(t))\right)+\O(N^{-\fc/9}),
\end{align}

In this way we can express the locations of eigenvalues in terms of the integrals of the Stieltjes transform of the empirical eigenvalue densities.  We have,
\begin{align}\begin{split}\label{e:sandwich}
&\phantom{{}={}}\bE_{H(t)}\left[\prod_{i=1}^k K_i\left(\Im\left[\frac{N}{\pi}\int_{s_iN^{-2/3}+\ell}^{N^{-2/3+\fc}} m_t(\tilde L_t+y+\ri \eta)\right]\rd y\right)\right]+\O\left(N^{-\fc/9}\right)\\
&\leq \bP_{H(t)}\left( N^{2/3} ( \lambda_i(t) - L_t - \cX )\geq s_i,1\leq i\leq k \right) =\bE\left[\prod_{i=1}^k{\bf 1}_{\cN_t(s_iN^{-2/3})\geq i}\right] \\
&\leq \bE_{H(t)}\left[\prod_{i=1}^k K_i\left(\Im\left[\frac{N}{\pi}\int_{s_iN^{-2/3}-\ell}^{N^{-2/3+\fc}} m_t(\tilde L_t+y+\ri \eta)\right]\rd y\right)\right]+\O\left(N^{-\fc/9}\right).\end{split}
\end{align}

For the product of the functions of Stieltjes trasnform, we have the following comparison theorem.

\begin{proposition}\label{p:comp}
Let $H$ be as in Definition \ref{asup}, with $q\gg N^{1/9}$. We 
fix $\fc>0$, $E_1,E_2,\cdots, E_k=\O( N^{-2/3})$, $\eta=N^{-2/3+\fc}$ and $F:\bR^k\mapsto \bR$ a bounded test function with bounded derivatives. For $t\ll1$ we have
\begin{align}\begin{split}\label{e:comparison} 
&\phantom{{}={}}\bE_{H}\left[F\left(\left\{\Im\left[N\int_{E_i}^{N^{-2/3+\fc}} m(\tilde L+y+\ri \eta)\right]\right\}_{i=1}^k\right)\right]\\
&=\bE_{H(t)}\left[\left\{F\left(\Im\left[N\int_{E_i}^{N^{-2/3+\fc}} m_t(\tilde L_t+y+\ri \eta)\right]\right\}_{i=1}^k\right)\right]+\O\left( N^{6\fc}\left(\frac{N^{2/3} t}{q^3}+\frac{N^{1/3}t}{q}+\frac{t^{1/2}N^{1/6}}{q}\right)\right).
\end{split}\end{align}
In the special case, $q=C N^{1/6}$, we have
\begin{align}\begin{split}\label{e:comparison2} 
&\phantom{{}={}}\bE_{H}\left[F\left(\left\{\Im\left[N\int_{E_i}^{N^{-2/3+\fc}} m(L+y+\ri \eta)\right]\right\}_{i=1}^k\right)\right]\\
&=\bE_{H(t)}\left[\left\{F\left(\Im\left[N\int_{E_i}^{N^{-2/3+\fc}} m_t(L_t+y+\ri \eta)\right]\right\}_{i=1}^k\right)\right]+\O\left( N^{6\fc}\left(\frac{N^{2/3} t}{q^3}+\frac{N^{1/3}t}{q}+\frac{t^{1/2}N^{1/6}}{q}\right)\right).
\end{split}\end{align}
\end{proposition}

\begin{proof}[Proof of Theorem \ref{thm:comp}]
Since $q\gg N^{1/9}$ and $t=N^{-1/3+\fd}$, we can take $\fc$  and $\fd$ small, and then the error terms in \eqref{e:comparison} and \eqref{e:comparison2} are of order $\O(N^{-\fc})$.  By combining \eqref{e:sandwich} and \eqref{e:comparison}, we get
\begin{align}\begin{split}
&\leq \bP_{H(t)}\left( N^{2/3} ( \lambda_i(t) - L_t - \cX )\geq s_i+2N^{2/3}\ell,1\leq i\leq k \right)+\O(N^{-\fc/9})\\
&\leq \bE_{H(t)}\left[\prod_{i=1}^k K_i\left(\Im\left[\frac{N}{\pi}\int_{s_iN^{-2/3}+\ell}^{N^{-2/3+\fc}} m_t(\tilde L_t+y+\ri \eta)\right]\rd y\right)\right]+\O\left(N^{-\fc/9}\right)\\
&\leq \bP_{H}\left( N^{2/3} ( \lambda_i(0) - L - \cX )\geq s_i,1\leq i\leq k \right) \\
&\leq \bE_{H(t)}\left[\prod_{i=1}^k K_i\left(\Im\left[\frac{N}{\pi}\int_{s_iN^{-2/3}-\ell}^{N^{-2/3+\fc}} m_t(\tilde L_t+y+\ri \eta)\right]\rd y\right)\right]+\O\left(N^{-\fc/9}\right)\\
&\leq \bP_{H(t)}\left( N^{2/3} ( \lambda_i(t) - L_t - \cX )\geq s_i-2N^{2/3}\ell,1\leq i\leq k \right)+\O(N^{-\fc/9}).
\end{split}
\end{align}
Since $N^{2/3}\ell=N^{-\fc/3}\ll1$, \eqref{e:comp1} follows. For \eqref{e:comp2}, analogous to \eqref{e:sandwich}, we have,
\begin{align}\begin{split}\label{e:sandwich2}
&\phantom{{}={}}\bE_{H(t)}\left[\prod_{i=1}^k K_i\left(\Im\left[\frac{N}{\pi}\int_{s_iN^{-2/3}+\ell}^{N^{-2/3+\fc}} m_t(L_t+y+\ri \eta)\right]\rd y\right)\right]+\O\left(N^{-\fc/9}\right)\\
&\leq \bP_{H(t)}\left( N^{2/3} ( \lambda_i(t) - L_t )\geq s_i,1\leq i\leq k \right)  \\
&\leq \bE_{H(t)}\left[\prod_{i=1}^k K_i\left(\Im\left[\frac{N}{\pi}\int_{s_iN^{-2/3}-\ell}^{N^{-2/3+\fc}} m_t(L_t+y+\ri \eta)\right]\rd y\right)\right]+\O\left(N^{-\fc/9}\right).\end{split}
\end{align}
Hence \eqref{e:comp2} follows from combining \eqref{e:sandwich2} and \eqref{e:comparison2}.

\end{proof}

\begin{proof}[Proof of Proposition \ref{p:comp}]
For simplicity of notation we only prove the case $k=1$.  The general case can be proved in the same way. Let,
\begin{align}
X_t=\Im\left[N\int_{E}^{N^{-2/3+\fc}}m_t(\tilde L_t +y+\ri \eta) \rd y\right].
\end{align}
We prove that 
\begin{align}\label{e:onecomp}
|\bE[F(X_t)]-\bE[F(X_0)]|\lesssim N^{6\fc}\left(\frac{N^{2/3} t}{q^3}+\frac{N^{1/3}t}{q}+\frac{t^{1/2}N^{1/6}}{q}\right).
\end{align}
Similarly to \cite[Proposition 7.2]{Lee2016}
\begin{align}
\frac{\rd}{\rd t}\bE[F(X_t)]
=\bE\left[F'(X_t)\frac{\rd X_t}{\rd t}\right]
=\bE\left[F'(X_t)\Im\int_{E}^{N^{-2/3+\fc}}\left(\sum_{ijk}\dot{h}_{jk}(t)\frac{\del G_{ii}}{\del H_{jk}}+(\dot{L}_t+\dot{\cX}_t)\sum_{ij}G_{ij}^2\right)\rd y\right],
\end{align}
where by definition
\begin{align}
\dot{h}_{jk}(t)=-\frac{1}{2}e^{-t/2}h_{jk}+\frac{e^{-t}}{2\sqrt{1-e^{-t}}}w_{jk}, \quad \dot{\cX}(t)=\frac{2}{N}\sum_{ij}h_{ij}(t)\dot{h}_{ij}(t).
\end{align}
By a large deviation estimate, we have 
$
|\dot\cX_t|\prec 1/(t^{1/2}N^{1/2}q),
$
and 
\begin{align}
\bE\left[F'(X_t)\Im\int_{E}^{N^{-2/3+\fc}}\dot{\cX}_t\sum_{ij}G_{ij}^2\rd x\right]\prec \frac{N^{1/6}}{t^{1/2}q}.
\end{align}
By the cumulant expansion formula, we get
\begin{align}
-\sum_{ijk}\bE\left[\dot h_{jk}(t)F'(X_t) G_{ij}G_{ki}\right]
=\sum_{p\geq 2} \frac{e^{-(p+1)t}\cC_{p}}{2Nq^{p-1}}\bE[\del_{jk}^p(F'(X_t)G_{ij}G_{ki})]
\end{align}
We notice that 
\begin{align}
\dot L_t=\mp\frac{12\cC_4}{q^2}+\O\left(\frac{1}{q^4}\right)
\end{align}
Thus, \eqref{e:onecomp} follows from the following Proposition. The proof of \eqref{e:comparison2} follows from the same argument as above, with replacing $\tilde L_t$ by $L_t$.
\end{proof}

\begin{proposition}\label{p:Dterms}
Under the assumptions of Proposition \ref{p:comp}, for any $p\geq 2$, let,
\begin{align}
J_p\deq  \frac{e^{-(p+1)t}\cC_{p+1}}{2Nq^{p-1}}\bE[\del_{jk}^p(F'(X_t)G_{ij}G_{ki})].
\end{align}
Then, with overwhelming probability
$
J_2=\O(N^{1+5\fc}/q),
$
\begin{align}
J_3=\pm\frac{12\cC_4}{q^2}\sum_{ij}\bE[F'(X_t)G_{ij}^2]+\O\left(\frac{N^{4/3+2\fc}}{q^3}+\frac{N^{1+5\fc}}{q^2}\right)
\end{align}
and for any $p\geq 4$, $J_p=\O(N^{4/3+2\fc}/q^3)$.
\end{proposition}
\begin{proof}[Proof of Proposition \ref{p:Dterms}]
Let $\tilde z_t=\tilde L_t+y+\ri \eta$.  We have,
\begin{align}\label{e:derofm}
|\del_{ij}^p\tilde z_t|\lesssim \frac{1}{N},
\quad
|\del^p_{ij}\tilde m_t(\tilde z_t)|\lesssim \frac{\Im[\tilde m_t(\tilde z_t)]}{N\eta}\lesssim \frac{\sqrt{y+\eta}}{N\eta}\lesssim N^{-2/3+3\fc/2},\quad p\geq 1.
\end{align}
and 
\begin{align}\label{e:derofGG}
\frac{1}{N^3}\sum_{ijk}\left|\del_{jk}^p(G_{ij}(\tilde z_t)G_{ki}(\tilde z_t))\right|\lesssim \frac{\Im[\tilde m_t(\tilde z_t)]}{N\eta}\lesssim N^{-2/3+3\fc/2},\quad p\geq 0,
\end{align}
with overwhelming probability.
Therefore, we have that 
\begin{align}\label{e:derofXt}
|\del_{jk}^pX_t|
=\left| \Im\left[N\int_{E}^{N^{-2/3+\fc}}\del_{jk}^p\tilde m_t(\tilde L_t +y+\ri \eta) \rd y\right]\right|\lesssim N^{-1/3+5\fc/2},\quad p\geq 1.
\end{align}
with overwhelming probability.
For any monomial of Green's function with at least three off-diagonal terms, e.g.,
\begin{align}\begin{split}\label{e:3offdiag}
\frac{1}{N}\sum_{ijk}|G^2_{ij}G_{jk}G_{ki}|,\quad
 \frac{1}{N}\sum_{ijk}|G^2_{ij}G_{jk}|,\quad
  \frac{1}{N}\sum_{ijk}|G_{ij}G_{jk}G_{ki}|,\\
\frac{1}{N}\sum_{ijk}|G^2_{ij}G^2_{jk}|,\quad
\quad \frac{1}{N}\sum_{ijk}|G^2_{ij}G^3_{jk}|,\quad
\frac{1}{N}\sum_{ijk}|G_{ij}G_{ik}G_{jk}^3|
\end{split}\end{align}
are of order $\O(N^{1+3\fc})$ with overwhelming probability.

For any $p\geq 4$, using \eqref{e:derofGG} and \eqref{e:derofXt}, we have the following bound,
\begin{align}
J_p =\frac{\O(1)}{Nq^3}\sum_{ijk}\bE[\del_{jk}^p(F'(X_t)G_{ij}G_{ki})]
\lesssim \frac{N^{4/3+2\fc}}{q^3},
\end{align}
with overwhelming probability.

For $p=2$, with overwhelming probability, we have
\begin{align}\begin{split}
J_2 &=\frac{\O(1)}{Nq}\sum_{ijk}\bE[\del_{jk}^2(F'(X_t)G_{ij}G_{ki})]
=\frac{\O(1)}{Nq}\bE[F'(X_t)\del_{jk}^2(G_{ij}G_{ki})]+\O\left(\frac{N^{1+5\fc}}{q}\right)\\
&=\frac{\O(1)}{Nq}\sum_{ijk}\bE[F'(X_t)G_{ij}G_{ki}G_{jj}G_{kk}]+\frac{1}{q}\{\text{terms in \eqref{e:3offdiag}}\}+\O\left(\frac{N^{1+5\fc}}{q}\right)\\
&=\frac{\O(1)}{Nq}\sum_{ijk}\bE[F'(X_t)G_{ij}G_{ki}G_{jj}G_{kk}]+\O\left(\frac{N^{1+5\fc}}{q}\right)\\
\end{split}\end{align}
For the first term, we use the identity $G_{ij}=\sum_{\ell\neq j}G_{jj}h_{j\ell}G_{\ell i}^{(j)}$, and the cumulant expansion
\begin{align}\begin{split}
&\phantom{{}={}}\frac{\O(1)}{Nq}\sum_{ijk}\bE[F'(X_t)G_{ij}G_{ki}G_{jj}G_{kk}]
=\frac{\O(1)}{Nq}\sum_{ijk, \ell\neq j}\bE[F'(X_t)G_{jj}h_{j\ell}G_{\ell i}^{(j)}G_{ki}G_{jj}G_{kk}]\\
&=\sum_{p\geq 1}\frac{\O(1)}{N^2q^p}\sum_{ijk, \ell\neq i}\bE[G_{\ell i}^{(j)}\del_{j\ell}^p(F'(X_t)G_{jj}G_{ki}G_{jj}G_{kk})]=\O\left(\frac{N^{1+5\fc}}{q}\right).
\end{split}\end{align}
where in the last estimate we used \eqref{e:derofGG}, \eqref{e:derofXt} and \eqref{e:3offdiag}.

For $p=3$, we have
\begin{align}\begin{split}
&\phantom{{}={}}J_3 
=\frac{e^{-4t}\cC_4}{2Nq^2}\sum_{ijk}\bE[F'(X_t)\del_{jk}^3(G_{ij}G_{ki})]+\O\left(\frac{N^{1+5\fc}}{q^2}\right)\\
&=-\frac{12e^{-4t}\cC_4}{Nq^2}\sum_{ijk}\bE[F'(X_t)G_{ik}G_{jj}G_{kk}G_{jj}G_{ki})]+\frac{1}{q^2}\{\text{terms in \eqref{e:3offdiag}}\}+\O\left(\frac{N^{1+5\fc}}{q^2}\right)\\
&=-\frac{12e^{-4t}\cC_4}{Nq^2}\sum_{ijk}\bE[F'(X_t)G_{ik}G_{jj}G_{kk}G_{jj}G_{ki})]+\O\left(\frac{N^{1+5\fc}}{q^2}\right)\\
&=\frac{12e^{-4t}\cC_4}{q^2}\sum_{ik}\bE[F'(X_t)G_{ik}G_{ki})]+\O\left(\frac{N^{4/3+2\fc}}{q^3}+\frac{N^{1+5\fc}}{q^2}\right)\\
&=\frac{12\cC_4}{q^2}\sum_{ik}\bE[F'(X_t)G_{ik}G_{ki})]+\O\left(\frac{N^{4/3+2\fc}}{q^3}+\frac{N^{1+5\fc}}{q^2}\right),
\end{split}\end{align}
where in the second to last line we used that $G_{jj}=\tilde m_{t}(\tilde z)+\O_{\prec}(1/q+1/N\eta)=-1+\O_{\prec}(1/q)$, and \eqref{e:derofGG}. This finishes the proof of Proposition \ref{p:Dterms}.

\end{proof}

\appendix

\section{Ward Identity}
Let $H$ be a symmetric matrix, with Green's function $G$, then for any $z=E+\ri\eta$ with $\eta>0$, we have the following Ward identity:
\begin{align}\label{e:Ward}
\sum_{j}|G_{ij}(z)|^2=\frac{\Im[G_{ii}(z)]}{\eta}.
\end{align}

The following follows from local semicircle estimates for sparse random matrices and the eigenvector delocalization.
\begin{proposition}
Let $H$ be as in Definition \ref{asup} and fix a large constant $\fb>0$. Then uniformly for any $z=E+\ri\eta$ such that $-\fb\leq E\leq \fb$ and $1/N\ll\eta\leq \fb$,  we have
\begin{align}\label{e:offbound1}
\max_{ij}|G_{ij}|=\O(1),
\end{align}
with overwhelming probability and 
\begin{align}\label{e:offbound2}
\quad \max_{i}\Im[G_{ii}(z)]\prec \Im[m_N(z)]
\end{align}
where $G$ is the Green's function of the matrix $H$.
\end{proposition}

\section{Proof of Proposition \ref{p:minfty} and \ref{p:tminfty}}
\label{ap:pftminfty}

The proof of Proposition \ref{p:minfty} and \ref{p:tminfty} follow similarly to \cite[Lemma 4.1]{Lee2016}. We first proof Proposition \ref{p:tminfty}.
{
\begin{figure}[h]
\center
\includegraphics[height=4.5cm]{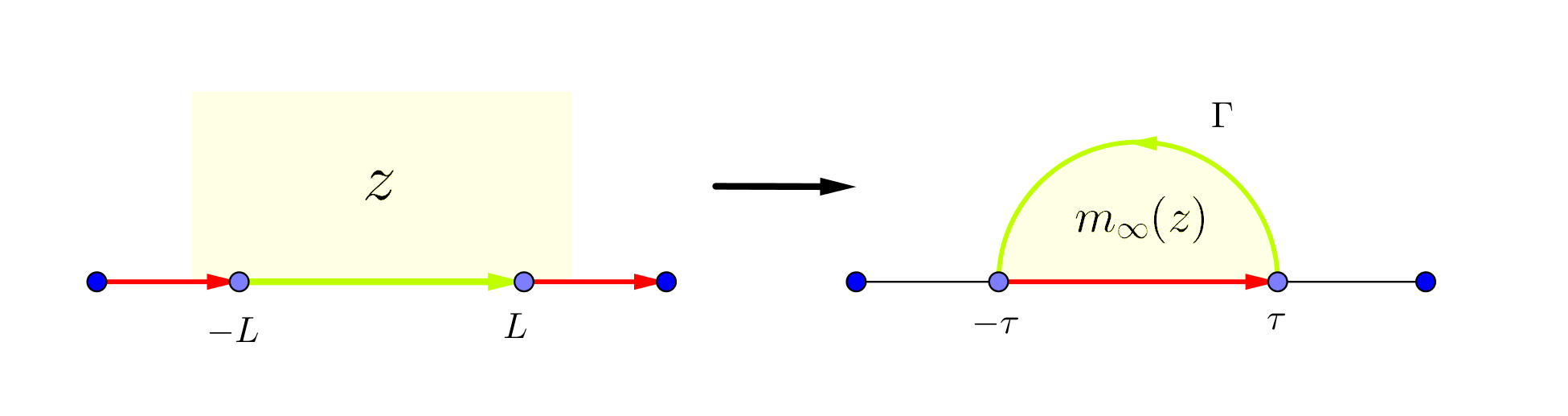}
\caption{$m_{\infty}$ is a biholomorphic map from the upper half plane to the region bounded by the curve $\Gamma$ and the interval $[-\tau, \tau]$, where the boundary $[-L,L]$ is mapped to the curve $\Gamma$, $[-\infty, -L]$ is mapped to $[0,\tau]$, and $[L,\infty]$ is mapped to $[-\tau, 0]$. }
\label{f:biholo}
\end{figure}
}
\begin{proof}[Proof of Proposition \ref{p:minfty}]
We introduce the following domains of the complex plane 
\begin{align}\begin{split}
\cD_z &\deq \{E+\ri\eta: -2\leq E\leq 2, 0\leq\eta\leq 2\},\\
\cD_w &\deq \{w\in \bC: |w|<3\}.
\end{split}\end{align}
Let 
\begin{align}\label{e:defR}
R(w)\deq -\frac{1}{w}-w-\frac{1}{q^2}\left(a_2w^3+a_3w^5+\cdots\right),
\end{align}
By definition, $P_0(z,w)=0$ if and only if $z=R(w)$. We first check that $R(w)$ has exactly two critical points on $\cD_w$. The derivatives of $R(w)$ are given by 
\begin{align}
R'(w)=\frac{1}{w^2}-1-\frac{1}{q^2}(3a_2w^2+5a_3w^4+\cdots),\quad R''(w)=-\frac{2}{w^3}-\frac{1}{q^2}(6a_2w+20a_3w^3+\cdots).
\end{align}
For $q$ large enough, $R''(w)<0$ on $(0,3)$ and $R''(w)>0$ on $(-3,0)$. Therefore $R'(w)$ is monotonically decreasing on $(0,3)$, and monotonically increasing on $(-3,0)$. Furthermore, for $C$ large enough, we have $R'(1-C/q^2)>0$ and $R'(1+C/q^2)<0$.  $R'(w)$ is an even function and so $R'(w)=0$ has two solutions on $(-3,3)$, which we denote by $\pm \tau$. Note that $\tau\in (1-C/q^2, 1+C/q^2)$. We let $L=R(-\tau)$, and so $-L=R(\tau)$.  We have that $L=2+\O(1/q^2)$. Next we show that $R(w)$ does not have other critical points on $\cD_w$.  The equation $R'(w)=0$ is equivalent to
\begin{align}\label{e:criticalpoints}
1-w^2-\frac{1}{q^2}\left(3a_2w^4+5a_3w^6+\cdots\right)=0.
\end{align}
On the boundary of $\cD_w$, we have
\begin{align}
|1-w^2|\geq 8\gg \frac{1}{q^2}\left|3a_2w^4+5a_3w^6+\cdots\right|.
\end{align}
Hence, by Rouch{\'e}'s theorem, the equation \eqref{e:criticalpoints} has the same number of roots as the quadratic equation $1-w^2=0$ on $\cD_w$. Since $1-w^2=0$ has two solutions on $\cD_w$, we find that $R(w)$ has exactly two critical points, i.e., $\pm \tau$ on $\cD_w$. 

We next show that for any $z\in \cD_z$, $P_0(z,w)$ has exactly two (counting multiplicity) solutions on $\cD_w$. On the boundary of $\cD_w$, we have
\begin{align}
|w^2+zw+1|\geq |w|^2-|z||w|-1\geq 2\gg \frac{1}{q^2}\left|a_2w^4+a_3w^6+\cdots\right|.
\end{align}
Hence, by Rouch{\'e}'s theorem, the equation $P_0(z,w)=0$ has the same number of roots as the quadratic equation $1+zw+w^2=0$ on $\cD_w$. Since $1+zw+w^2=0$ has two solutions on $\cD_w$, we find that $P_0(z,w)=0$ has exactly solutions (counting multiplicity)  on $\cD_w$. In particular, for any $z\in (-L,L)$, $P(z, w)=0$ has two solutions $w(z)\in \bC_+$ and $\bar w(z)\in \bC_-$, and $w(z)$ forms a curve on $\bC_+$, joining $-\tau$ and $\tau$, which we denote by $\Gamma$.

Then we that $R(w)$ is biholomorphic from the region $\cD_\Gamma$, bounded by the curve $\Gamma$ and the interval $[-L,L]$ to the upper half plane, as in Figure \ref{f:biholo}. By definition \eqref{e:defR}, $R(w)$ is holomorphic on $\cD_\Gamma$, and continuously extends to the boundary except at $w=0$. In a small neighborhood of $0$, the imaginary part of $R(w)$ satisfies
\begin{align}
\Im[R(w)]=\Im[w]\left(\frac{1}{|w|^2}-1\right)+\O\left(\frac{\Im[w]}{q^2}\right)>0.
\end{align}
Hence, by maximum principle, we have $\Im[R(w)]>0$ on $w\in \cD_\Gamma$, i.e. $R(w)$ maps $\cD_\Gamma$ to the upper half plane. If we view $R(w)$ as a map to the compactified complex plane, i.e. $\bC\cup\{\infty\}$, then $R(0)=\infty$. Moreover, it extends to a orientation preserving bijection between the boundary of $\cD_\Gamma$ and the boundary of $\cC_+$, where the boundary curve $\Gamma$ is mapped to $[-L,L]$, $[0,\tau]$ is mapped to $[-\infty,-L]$, and $[-\tau,0]$ is mapped to $[L, \infty]$. For any $z\in\bC_+$, the number of solutions to $z=R(w)$ is given by 
\begin{align}
\frac{1}{2\pi \ri}\int_{\del \cD_\Gamma}\frac{R'(w)}{R(w)-z}\rd w
=\frac{1}{2\pi \ri}\int_{\del \bC_+\cup \{\infty\}}\frac{1}{\xi-z}\rd \xi=1.
\end{align}
It follows $R(w)$ is a biholomorphic map from $\cD_\Gamma$ to the upper half plane.  We take $m_{\infty}$ to be the inverse of $R(w)$. 

Finally, we check the properties of $m_{\infty}(z)$. The function $m_{\infty}(z)$ is a holomorphic map from the upper half plane $\bC_+$ to $\cD_\Gamma$, satisfying $P_0(z, m_{\infty}(z))=0$. Let $\rho_\infty$ be the measure obtained by Stieltjes inversion of $m_\infty(z)$. To show that $\rho_\infty$ is a probability measure, it suffices to check that $\lim_{y\rightarrow \infty}\ri y m_{\infty}(\ri y) =-1$. Since $m_{\infty}(z)$ is bounded, one can check that $\lim_{z\rightarrow \infty} m_{\infty}(z)=0$. Thus,
\begin{align}
0=\lim_{y\rightarrow \infty}P_0(z, m_{\infty}(\ri y))=\lim_{y\rightarrow \infty}(1+\ri ym_{\infty}(\ri y)),
\end{align}
which implies that  $\lim_{y\rightarrow \infty}\ri y m_{\infty}(\ri y) =-1$, and $\rho_{\infty}$ is a probability measure. Moreover, 
\begin{align}
\lim_{\eta\rightarrow 0} \Im [m_{\infty}(E+\ri\eta)]=0,
\end{align}
for $E\not\in [-L,L]$. Hence, $\rho_{\infty}$ is a measure supported on $[-L,L]$. 

In the following we study the behaviors of $m_{\infty}$ and $\rho_{\infty}$ at the edges $\pm L$.  Let $z=R(w)$.  In a small neighborhood of the critical point $-\tau$, we have the expansion,
\begin{align}\begin{split}
z&=R(-\tau)+R'(-\tau)(w+\tau)+\frac{R''(-\tau)}{2}(w+\tau)^2(1+\cE(w))\\
  &=L+\frac{R''(-\tau)}{2}(w+\tau)^2(1+\cE(w)),
\end{split}\end{align}
where the error term satisfies $\Re[\cE(w)]=\O(|w+\tau|)$ and $\Im[\cE[w]]=\O(\Im[w])$.
Since $R''(-\tau)>0$, we find that in a small neighborhood of $-\tau$
\begin{align}\label{e:wtau}
\left(\frac{2}{R''(-\tau)}\right)^{1/2}\sqrt{z-L}=(w+\tau)\left(1+\cE(w)\right)^{1/2},
\end{align}
where we choose the branch of the square root so that $\sqrt{z-L}\in \bC^+$. By taking imaginary part on both sides of \eqref{e:wtau}, and noting  that $\Im[(1+\cE(w))^{1/2}]=\O(\Im[w])$, we get
\begin{align}
\Im[\sqrt{z-L}]\asymp \Im[w](1+\O(w+\tau)).
\end{align}
In a small neighborhood of $-\tau$, or equivalently, in a small neighborhood of $L$, we get
\begin{align}
\Im[m_{\infty}(z)]= \Im[w]\asymp\Im[\sqrt{z-L}]\asymp 
\left\{
\begin{array}{cc}
\sqrt{\kappa+\eta}, & \text{if $\Re[z]\leq L$,}\\
\eta/\sqrt{\kappa+\eta}, & \text{if $\Re[z]\geq L$,}
\end{array}
\right.
\end{align}
where $\kappa =|\Re[z]-z|$. In particular, by the Stieltjes inversion formula, $\rho_\infty$ has square root behavior at edge $L$.  For the derivative of $P_0(z,m_\infty(z))$, by the chain rule, we have
\begin{align}
0=\del_z P_0(z, m_\infty(z))=m_\infty(z)+\del_zm_\infty(z)\del_2P_0(z,m_{\infty}(z)).
\end{align}
Since $m_\infty(z)$ has square root behavior at $L$, we have $|\del_z m_\infty(z)|\asymp |z-L|^{-1/2}$, and
\begin{align}
|\del_2 P_0(z,m_{\infty}(z))|=\frac{|m_\infty(z)|}{|\del_z m_\infty(z)|}\asymp \sqrt{|z-L|}.
\end{align}
We can apply the same argument and get the same estimates of the behaviors of $m_\infty$ and $\rho_\infty$ in a small neighborhood of $-L$.
This finishes the proof of Proposition \ref{p:minfty}. 
\end{proof}

\begin{proof}[Proof of Proposition \ref{p:tminfty}]
We only prove the estimate \eqref{e:esttL}. The rest of Proposition \ref{p:tminfty} follows from the same argument as in Proposition \ref{p:minfty}. 
Similar to the proof of Proposition \ref{p:minfty}, we define $\tilde R(w)$, $\tilde \tau$ and $\tilde L=\tilde R(-\tilde \tau)$. Moreover, we have $\tilde R(w)=R(w)-\cX w$. By a perturbation argument we get,
\begin{align}
\tilde \tau =\tau +\frac{\cX}{R''(\tau)}+\O_{\prec}\left(\frac{1}{Nq^2}\right),\quad \tilde L =L+\cX\tau +\O_{\prec}\left(\frac{1}{Nq^2}\right).
\end{align}
Moreover, since $\tau=1+\O(1/q^2)$. The claim follows.
\end{proof}

\section{Free convolution calculations} \label{a:fc}
In this appendix we consider $H$ as in Definition \ref{asup}, with $q \gg N^{1/9}$.  We will calculate various free convolution quantities.  We let
\beq
t := \frac{ N^{\fd}}{N^{1/3}}
\eeq
We recall the construction of $L$, $\tilde L$, $\tilde \rho_\infty(x)$ and $\tilde m_\infty(z)$ from Proposition \ref{p:minfty} and \ref{p:tminfty}. We use the definitions of Section \ref{sec:ht}.  
\begin{lemma} \
For all $E \asymp t^2$ and $k \geq 1$, we have
\beq \label{eqn:hsest}
\left| \int \frac{1}{ (x - (\tilde{L} + E ) )^k } \d \mu_{H} (x) - \int \frac{1}{ (x- ( \tilde{L} + E ) )^k }  \tilde{\rho}_\infty (x)\d x \right| \prec \frac{1}{ t^{2(k-1)} }\left( \frac{1}{ N t^2} + \frac{1}{ N^{1/2} t q^{3/2} }\right).
\eeq
\end{lemma}
\begin{proof}
Let $\chi(y)$ be a smooth cut-off function such that $\chi (y) = 1$  for $|y| \leq t^2$ and $\chi (y) = 0$ for $|y| \geq 2 t^2$, such that $| \chi'(y)| \leq C t^{-2}$.  Let $f(x)$ be a cut-off function so that $f(x) = 1$ for $x \leq \tilL + E/3$ and $f(x) = 0$ for $x \geq \tilL + 2 E /3$.  We choose $f(x)$ so that the derivatives obey $| f^{(k)} (x)| \leq C t^{-2k}$.  Since by \eqref{e:largeeig} with overwhelming probability all the eigenvalues of $H$ are less than $\tilL + N^{-2/3+\varepsilon}$ for any $\varepsilon >0$, we can consider instead the test function
\beq
g(x) := f(x) \frac{1}{ (x- ( \tilL + E) )^k}.
\eeq
Let
\beq
S(x + \i y ) := m_N (x + \i y ) - \tilminf (x + \i y ).
\eeq
Define $\eta_* := N^{-2/3+\varepsilon}$ for $\varepsilon >0$, with $\varepsilon < \fd/4$, so that $\eta_* < t^2$.  Note that for $y \geq \eta_*$ we have by Theorem \ref{thm:edgerigidity} that
\beq
| S(x + \i y ) | \prec \frac{1}{N y } + \frac{1}{ ( N y )^{1/2} q^{3/2} } =: \Phi (y),
\eeq
for $q \geq N^{1/9}$ regardless of whether $x\geq \tilL$ or $x\leq \tilL$. 
By the Helffer-Sj{\"o}strand formula,
\begin{align}
\begin{split}
\left| \int g(x) \d \muHo (x) - \int g(x) \tilrhoinf (x) \d x \right| &\leq \left| \int g''(x) \chi (y)y \Im [ S (x + \i y ) ] \d x \d y \right| \\
&+\int | g(x) \chi' (y)  \Im [ S ( x + \i y ) ] | \d x \d y \\
&+ \int | g' (x) \chi' (y) y \Re [ S ( x + \i y ) ] | \d x \d y.
\end{split}\end{align}
For the second two terms note that
\beq
\int |g(x) | \d x \leq \frac{C \log(N) }{t^{2k-2}},\quad \int |g'(x) | \d x \leq \frac{C}{ t^{2k}}, \quad \int |g''(x) | \d x \leq \frac{C }{t^{2k+2}},
\eeq
and that the $y$ integral is only over $t^2< |y| < 2 t^2$.
So, 
\beq
\int | g(x) \chi' (y)  \Im [ S ( x + \i y ) ] | \d x \d y + \int | g' (x) \chi' (y) y \Re [ S ( x + \i y ) ] | \d x \d y \prec \frac{  \Phi ( t^2 )}{t^{2k-2}}.
\eeq
For the last term we first use the fact that $y \to y \Im [ m_N]$ is increasing to bound the $|y| \leq \eta_*$ integral by
\begin{align}
\left| \int_{ |y| \leq \eta_*} g''(x) \chi (y) y \Im [ S (x + \i y ) ]  \d x \d y \right| &\leq 2 \int_{y \leq \eta_*} |g''(x) | \eta_* \Im [ \tilminf (x + \i \eta_* ) ] \d x \d y \notag\\
&+ \int_{y \leq \eta_*} |g''(x) | \eta_* |S(x+\ri \eta_* ) |\d x \d y \label{eqn:hs1}
\end{align}
The second term is  bounded by
\beq
\int_{y \leq \eta_*} |g''(x) | \eta_* |S(x+\ri \eta_* ) | \d x \d y \prec \frac{\Phi ( \eta_* ) \eta_*^2}{ t^{2k+2}} \lesssim \frac{ \Phi ( t^2 )}{ t^{2k-2} }
\eeq
where we used $\eta_* \leq t^2$.   We now estimate the first term of \eqref{eqn:hs1}.    Using \eqref{eqn:imtilminf} we find,
\begin{align}
\begin{split}
 \int_{y \leq \eta_*} |g''(x) | \eta_* \Im [ m (x + \i \eta_* ) ] \d x \d y  & \lesssim \int_{x < 0} |g''(x+ \tilL) |  \eta_*^2 \sqrt{ \eta_*+|x|} \d x \\
 &+   \int_{ x \geq 0} |g''(x+ \tilL) | \frac{ \eta_*^3}{ \sqrt{ \eta_* + x} } \d x.\label{eqn:hs2}
\end{split}\end{align}
We note that $|g''(x+\tilde L)|\lesssim C|x-E|^{-(k+2)}1_{x\leq 2E/3}$. The second integral is bounded by
\begin{align}\begin{split}
\int_{ x\geq 0} |g''(x+\tilL ) | \frac{ \eta_*^3}{ \sqrt{ \eta_* + x} } \rd x&\leq \int_{ 0 < x < \eta_*} |g''(x + \tilL) | \frac{ \eta_*^3}{ \sqrt{ \eta_* + x} } \rd x + \int_{ x \geq \eta_*} |g''(x + \tilL) | \frac{ \eta_*^3}{ \sqrt{ x} }\rd x \\
& \lesssim \frac{ \eta_*^{3.5}}{ t^{2k+4} } + \int_{x \geq \eta_*} |g''(x + \tilL) | \frac{ \eta_*^3}{ \sqrt{x}} \d x
\end{split}\end{align}
The last term is bounded by
\begin{align}\begin{split}
\int_{x \geq \eta_*} |g''(x + \tilL) | \frac{ \eta_*^3}{ \sqrt{x}} \d x &\lesssim \int_{ \eta_* \leq x <  E/3} \frac{1}{ |x -E|^{k+2} } \frac{\eta_*^3}{ \sqrt{x}} \d x + \int_{E/3 < x < 2E/3} \frac{1}{t^4 |x-E|^k } \frac{\eta_*^3}{ \sqrt{x}} \d x \\
&\lesssim \frac{ \eta_*^3}{ t^{2k+3} } 
\end{split}\end{align}
For the other term on the righthand side  of \eqref{eqn:hs2}, similarly we get
\beq
 \int_{x <0} |g''(x+\tilL) |  \eta_*^2 \sqrt{ \eta_*+|x|}  \d x \lesssim  \frac{ \eta_*^{2.5}}{ t^{2k+2}} +  \frac{ \eta_*^2}{ t^{2k+1}}
\eeq
In summary, we have the estimate,
\beq
\left| \int_{ y\leq \eta_*} g''(x) \chi (y) y \Im [ S (x + \i y ) ]  \d x \d y \right|  \prec  \frac{ \Phi (t^2)}{t^{2k-2}} + \frac{1}{ t^{2k-2}} \frac{ \eta_*^2 }{ t^{3}} \prec \frac{ \Phi (t^2)}{t^{2k-2}}
\eeq
where we used that $\eta_*^2 /t^3 = N^{2 \varepsilon } / ( N^{4/3} t^3 ) \leq 1 / ( N t^2)$ by our choice of $\varepsilon$ above.  Finally,
\begin{align}
\int_{y > \eta_* } |g'' (x) y \chi (y) \Im [ S (x+ \i y ) ] \d x \d y &\lesssim \frac{1}{t^{2k+2}} \int_{|y| \leq 2 t^2} y \Phi(y)  \d y  \lesssim  \frac{ \Phi ( t^2 ) }{ t^{2k-2}},
\end{align}
where we used the fact that the support of $\chi$ is in $|y| \leq 2 t^2$.  Putting together our estimates yields the claim.
\end{proof}

\begin{proof}[Proof of Lemma \ref{lem:freeconv1}] From the square root behaviour of $\tilminf$ we easily conclude that $\e^{t/2} \tilxip - \tilL \asymp t^2$.  Then, using \eqref{eqn:hsest} we conclude the same for $\hatxip$.  By taking difference of the defining equations \eqref{e:defxi} for $\tilxip$ and $\hatxip$, we get  
\beq \label{eqn:xisub}
\int \frac{\d \mu_H (x)}{ (x - \e^{t/2} \hatxip )^2}  - \int \frac{ \tilrhoinf (x) \d x }{ ( x - \e^{t/2} \hatxip )^2} =\int\left( \frac{1}{ ( x -\e^{t/2} \tilxip )^2 } - \frac{1}{ (x- \e^{t/2} \hatxip)^2 } \right)  \tilrhoinf (x) \d x
\eeq
We calculate the righthand side
\begin{align}\begin{split}
\int&\left( \frac{1}{ ( x -\e^{t/2} \tilxip )^2 } - \frac{1}{ (x- \e^{t/2} \hatxip)^2 } \right)  \tilrhoinf (x) \d x =  \int  \frac{\e^{t/2} (\tilxip - \hatxip ))( 2x -\e^{t/2} \tilxip -\e^{t/2} \hatxip)}{ (x- \e^{t/2} \tilxip )^2 ( x - \e^{t/2} \hatxip )^2} \tilrhoinf (x) \d x \\
&=  \e^{t/2}( \tilxip - \hatxip )\int  \frac{1}{ ( x - \e^{t/2} \hatxip ) ( x - \e^{t/2}\tilxip)^2} + \frac{1}{ ( x - \e^{t/2}\tilxip ) ( x - \e^{t/2} \hatxip)^2} \tilrhoinf (x ) \d x.
\end{split}\end{align}
Since both $\e^{t/2} \hatxip -\tilL \asymp \e^{t/2} \tilxip - \tilL \asymp t^2$, we see that the integrand is negative and moreover by the square root behaviour of $\tilrhoinf$,
\beq
\left| \int  \frac{1}{ ( x - \e^{t/2} \hatxip ) ( x - \e^{t/2}\tilxip)^2} + \frac{1}{ ( x - \e^{t/2}\tilxip )^2 ( x - \e^{t/2} \hatxip)} \tilrhoinf (x ) \d x \right| \asymp \frac{1}{t^3}.
\eeq
We immediately conclude \eqref{eqn:xiest} from this and the estimate \eqref{eqn:hsest} for the lefthand side of \eqref{eqn:xisub}.   Next, we estimate
\beq
| \hatEp - \tilEp | \leq | \hatxip - \tilxip | + C t | \tilde{m}_\infty (\e^{t/2} \tilxip ) - m_N ( \e^{t/2} \hatxip ) |.
\eeq
The latter quantity we estimate by 
\begin{align}\begin{split}
| \tilde{m}_\infty (\e^{t/2} \tilxip ) - m_N ( \e^{t/2} \hatxip ) | &\leq | \tilde{m}_\infty (\e^{t/2} \tilxip ) - \tilde{m}_\infty ( \e^{t/2} \hatxip ) | + | \tilde{m}_{\infty} (\e^{t/2} \hatxip ) - m_N ( \e^{t/2} \hatxip ) | \\
&\prec t\left( \frac{1}{N^{1/2} t^2 q^{3/2} } + \frac{1}{ N t^3} \right),
\end{split}\end{align}
where we used \eqref{eqn:hsest}.  
Hence, we have proven \eqref{eqn:edgeest}.   Finally, the estimate \eqref{eqn:scalest} is an easy consequence of \eqref{eqn:xiest} and \eqref{eqn:hsest}.  \end{proof}

\begin{lemma}\label{l:eloc}
Let the polynomial of a self-consistent equation be given by
\begin{align}\label{e:selfeqn}
0 = 1 + z m + m^2 + \frac{a}{q^2} m^4 + \frac{b}{q^4} m^6 +\cX m^2+P_1 (m),
\end{align}
where $P_1(m)$ is a polynomial in $m$ of size $\O(q^{-6})$, and constants $a,b=\O(1)$, and $m$ be the Stieltjes transform of a probability measure, supported on $[-\tilL,\tilL]$, with
\begin{align}
\tilde L=2 + \frac{a}{q^2}  + \frac{b}{q^4} - \frac{9 a^2}{4 q^4}+ \cX+\O ( \cX q^{-2} + \cX^2 + q^{-6} ).
\end{align}
For any $z$ in a small neighborhood of $\tilL$, $m(z)$ is given by 
\beq
m (z)= -\tau + \left( \frac{2}{F'' ( -\tau ) } \right)^{1/2} ( z-\tilL)^{1/2} + \O ( |z- \tilL|),
\eeq
where 
\beq
\tau = 1  - \frac{3 a }{q^2} - \frac{5b}{q^4} + \frac{ 18 a^2}{q^4} - \cX+ \O ( \cX q^{-2} + \cX^2 + q^{-6} ).
\eeq
\end{lemma}
\begin{proof}
We have the polynomial
\beq
0 = 1 + z m + (1+\cX) m^2 + \frac{a}{q^2} m^4 + \frac{b}{q^4} m^6 +P_1 (m)
\eeq
Solving  for $z$,
\beq
z = -\frac{1}{m} - m - \frac{a}{q^2} m^3 - \frac{b}{q^4} m^5-\cX m - P_2(m) := F(m),
\eeq
where $P_2(m)=P_1(m)/m$ is another polynomial of size $\O(q^{-6})$. There is a unique positive solution to $F' (-\tau) = 0$.  We record the expressions,
\beq
F' (m) = \frac{1}{m^2} - 1 -\frac{3 a }{q^2} m^2 - \frac{ 5 b}{q^4} m^4 -\cX- P_2' (m).
\eeq
and
\beq
F''(m) = -\frac{2}{m^3} - \frac{6 a}{q^2} m - \frac{ 20 b}{q^4} m^3 - P_2'' (m).
\eeq
We first find
\beq
\tau^2 = 1 - \cX - \frac{3 a }{q^2} - \frac{5b}{q^4} + \frac{ 18 a^2}{q^4} + \O ( \cX q^{-2} + \cX^2 + q^{-6} ),
\eeq
and subsequently,
we can expand to obtain,
\beq
\tau = 1 - \frac{ 3 a}{2 q^2}  - \frac{5b}{2 q^4} + \frac{63 a^2}{8 q^4 }- \frac{ \cX}{2} + \O ( \cX q^{-2} + \cX^2 + q^{-6} ).
\eeq
Substituting this in $F (- \tau)$ we get
\beq
\tilL = 2 + \frac{a}{q^2}+ \frac{b}{q^4} - \frac{9 a^2}{4 q^4}+ \cX +\O ( \cX q^{-2} + \cX^2 + q^{-6} ).
\eeq
We see that
\beq\label{e:F''}
F''( -\tau ) = 2 + \O ( q^{-2} ).
\eeq
We expand
\beq
z = \tilL + \frac{ F'' (- \tau ) }{ 2} ( m- \tau)^2 + \O ( | m- \tau |^3)
\eeq
and invert,
\beq\begin{split}
m(z) &= -\tau + \left( \frac{2}{F'' ( -\tau ) } \right)^{1/2} (z-\tilL)^{1/2} + \O ( |z- \tilL|).
\end{split}\eeq

\end{proof}

\begin{proof}[Proof of Lemma \ref{lem:freeconv2}]
In the following proof we write $m=\tilde m_{\infty}$, $\xi=\tilde\xi_+$ and $E_+=\tilde E_+$. Recall from \eqref{e:ms-ceqn}, $m$ satisfies a self-consistent equation, 
\begin{align}\label{e:selfeqn}
0 = 1 + z m + m^2 + \frac{6\cC_4}{q^2} m^4 + \frac{120\cC_6}{q^4} m^6 +\cX_t m^2+\O(q^{-6}).
\end{align}
Then by Lemma \ref{l:eloc}, we have
\begin{align}
\tilde L=2 + \frac{6\cC_4}{q^2}  + \frac{120\cC_6}{q^4} - \frac{81\cC_4^2}{q^4}+ \cX+\O ( \cX q^{-2} + \cX^2 + q^{-6} ).
\end{align}

In the following we first compute $E_+ = \xi - (1 - \e^{-t} )\e^{t/2} m ( \e^{t/2} \xi )$.
For $\xi$ (as defined in \eqref{e:defxi}), by Lemma \ref{l:eloc} and \eqref{e:F''}, we have,
\beq\begin{split}
\frac{\e^{-t}}{ 1- \e^{-t} } = m' ( \e^{t/2} \xi ) 
&= \frac{1}{ ( 4+\O(q^{-2}) )^{1/2} } \frac{1}{ \sqrt{ \e^{t/2} \xi - \tilL }} + \O(1).
\end{split}\eeq
We can solve for $\xi$ and get
\beq
\e^{t/2} \xi = \tilL + \frac{1}{4} (\e^t -1 )^2 + \O ( t^3 + t^2 q^{-2} ).
\eeq
 Using Lemma \ref{e:ms-ceqn}, we compute,
\begin{align}\begin{split}
m ( \e^{t/2} \xi ) &= -1+\frac{18\cC_4}{q^2}+ \left(1+\O(q^{-2}) \right)^{1/2} (\e^{t/2} \xi - \tilL )^{1/2} + \O (t^2) \\
&= -1 + \frac{9\cC_4}{ q^2} + \frac{ \e^{t}-1}{2} + \O\left(t^2 +q^{-4} \right).
\end{split}\end{align}
Hence,
\begin{align}\begin{split}\label{e:eqnE}
E_+ 
&= \e^{-t/2} \tilL + \frac{e^{-t/2} (\e^t -1 )^2}{4} - (1- \e^{-t} )\e^{t/2} \left( -1 + \frac{9\cC_4}{q^2} + \frac{ \e^t-1}{2} \right) + \O \left( t^3 + t q^{-4} + t^2 q^{-2} \right)\\
&= 2+
\frac{6(1-2t)\cC_4}{q^2} 
+ \frac{ 120 \cC_6}{q^2}- \frac{81\cC_4^2}{q^4} +\cX +\O \left( \cX q^{-2} + \cX^2 + q^{-6}+t \cX + t^3 + t q^{-4} + t^2 q^{-2} \right)\\
&= 2+
\frac{6(1-2t)\cC_4}{q^2} 
+ \frac{ 120 \cC_6}{q^2}- \frac{81\cC_4^2}{q^4} +\cX\\
&+\O_\prec \left(t^2q^{-2}+tq^{-4} + q^{-6} +tN^{-1/2}q^{-1}+N^{-1/2}q^{-3}+t^{1/2}N^{-1}\right).
\end{split}
\end{align}
where we used that $\cX\prec N^{-1/2}q^{-1}$.

The $k$-th cumulant of the entries of $H(t)$ (as defined in \eqref{e:Ht}) is $1/N$ when $k=2$, and 
\begin{align}
\frac{(k-1)!e^{-tn/2}\cC_k}{Nq^{k-2}},
\end{align}
for $k\geq 3$.
The self-consistent equation for $H(t)$ is given by
\begin{align}\label{e:selfeqn}
0 = 1 + z m_t + m_t^2 + \frac{6e^{-2t}\cC_4}{q^2} m_t^4 + \frac{120e^{-3t}\cC_6}{q^4} m_t^6 +\cX_t m_t^2+\O(q^{-6}),
\end{align}
By Lemma \ref{l:eloc}, we have
\begin{align}\label{e:estLt}
\tilL_t=2 + \frac{6e^{-2t}\cC_t}{q^2}  + \frac{120e^{-3t}\cC_6}{q^4} - \frac{81 e^{-4t}\cC_4}{ q^4}+ \cX_t+\O ( \cX_t q^{-2} + \cX_t^2 + q^{-6} )
\end{align}
We note that 
\beq
\cX_t = \cX + \O_\prec\left( t N^{-1/2} q^{-1} + t^{1/2}N^{-1} \right)=\O_\prec(N^{-1/2}q^{-1}),
\eeq
and rewrite \eqref{e:estLt} as
\begin{align}\begin{split}\label{e:eqnL}
\tilL_t&=
2  + \frac{ 6 \mathcal{C}_4}{q^2}(1 - 2 t)  + \frac{ 120 \mathcal{C}_6}{q^4} - \frac{81 \mathcal{C}_4^2}{q^4} + \cX\\
& +\O_\prec \left(t^2q^{-2}+tq^{-4} + q^{-6} +tN^{-1/2}q^{-1}+N^{-1/2}q^{-3}+t^{1/2}N^{-1}\right).
\end{split}\end{align}
Lemma \ref{lem:freeconv2} follows from comparing \eqref{e:eqnE} and \eqref{e:eqnL}, and recalling $q\gg N^{1/9}$.
\end{proof}

\bibliography{References.bib}{}
\bibliographystyle{abbrv}

\end{document}